\documentclass[12pt,english]{amsart}

\input xy
\xyoption{all}

\usepackage[latin1]{inputenc}
\usepackage{amsfonts,amsmath,amssymb,amsthm}
\usepackage{comment}

{\theoremstyle{plain}
\newtheorem{theorem}{Theorem}[section]    

\newtheorem{twisting lemma}[theorem]{Twisting lemma}

\newtheorem{lemma}[theorem]{Lemma}       
\newtheorem{proposition}[theorem]{Proposition}  
\newtheorem{corollary}[theorem]{Corollary}   

}
{\theoremstyle{remark}
\newtheorem{definition}[theorem]{Definition}      
\newtheorem{remark}[theorem]{Remark}   

}

\def\sep{{\scriptsize\hbox{\rm sep}}}

\def\cm{\hbox{\hbox{\rm C}\kern-5pt{\raise 1pt\hbox{$|$}}}}

\def\lhfl#1#2{\smash{\mathop{\hbox to 12mm{\leftarrowfill}}
\limits^{#1}_{#2}}}

\def\rhfl#1#2{\smash{\mathop{\hbox to 12mm{\rightarrowfill}}
\limits^{#1}_{#2}}}

\def\build#1_#2^#3{\mathrel{
\mathop{\kern 0pt#1}\limits_{#2}^{#3}}}

\def\htrait#1#2{\smash{\mathop{\hbox to 12mm{\hrulefill}}
\limits^{#1}_{#2}}}

\def\sxbullet{{\raise 2pt\hbox{\bf .}}}
\begin{document}

\title{Parametric Galois extensions} 

\author{Fran\c cois Legrand}

\email{Francois.Legrand@math.univ-lille1.fr}

\address{Laboratoire Paul Painlev\'e, Math\'ematiques, Universit\'e Lille 1, 59655 Villeneuve d'Ascq Cedex, France}

\date{\today}

\begin{abstract}
Given a field $k$ and a finite group $H$, {\it{an $H$-parametric extension over $k$}} is a finite Galois extension of $k(T)$ of Galois group containing $H$ which is regular over $k$ and has all the Galois extensions of $k$ of group $H$ among its specializations. We are mainly interested in producing non $H$-parametric extensions, which relates to classical questions in inverse Galois theory like the Beckmann-Black problem and the existence of one parameter generic polynomials. We develop a general approach started in a preceding paper and provide new non parametricity criteria and new examples.
\end{abstract}

\maketitle

\section{Presentation}

The {\it{Inverse Galois Problem}} asks whether, for a given finite group $H$, there exists at least one Galois extension of $\mathbb{Q}$ of group $H$. A classical way to obtain such an extension consists in producing a Galois extension $E/\mathbb{Q}(T)$ with the same group which is regular over $\mathbb{Q}$ \footnote{{\it{i.e.}} $E \cap \overline{\mathbb{Q}} = \mathbb{Q}$. See \S2.1 for basic terminology.}: from the {\it{Hilbert irreducibility theorem}}, $E/\mathbb{Q}(T)$ has at least one specialization of group $H$ (in fact infinitely many if $H$ is not trivial).
 
In this paper we are interested in ``parametric Galois extensions", {\it{i.e.}} in finite Galois extensions $E/\mathbb{Q}(T)$ which are regular over $\mathbb{Q}$ - from now on, say for short that $E/\mathbb{Q}(T)$ is a ``$\mathbb{Q}$-regular Galois extension" - and which have all the Galois extensions of $\mathbb{Q}$ of group $H$ among their specializations. More precisely, given a field $k$ and a finite group $H$, we say that a $k$-regular finite Galois extension $E/k(T)$ of group $G$ containing $H$ (with possibly $H \not=G$) is {\it{$H$-parametric over $k$}} if any Galois extension of $k$ of group $H$ ocurs as a specialization of $E/k(T)$ (definition \ref{ext st para}). The special case $H=G$ is of particular interest.

This was introduced in our previous paper \cite{Leg13a} in the number field case. Given a field $k$ and a finite group $G$, the question of whether there is a $G$-parametric extension over $k$ of group $G$ or not is intermediate between these classical two questions in inverse Galois theory:

\vspace{0.5mm}

\noindent
-  if there is such an extension, then it obviously solves the {\it{Beckmann-Black problem for $G$ over $k$}}, which asks whether any Galois extension $F/k$ of group $G$ occurs as a specialization of some $k$-regular Galois extension $E_F/k(T)$ with the same group,

\vspace{0.5mm}

\noindent
- if there are no such extension, then there obviously cannot exist a {\it{one parameter generic polynomial over $k$ of group $G$}}, {\it{i.e.}} a polynomial $P(T,Y) \in k(T)[Y]$ of group $G$ such that the splitting extension over $L(T)$ is $G$-parametric over $L$ for any field extension $L/k$.

\vspace{0.5mm}

\noindent
We refer to $\S$2.2 for more details.

If studying parametric extensions indeed seems a natural first step to these important topics, it is itself already quite challenging, especially over number fields. The question of deciding whether a given $k$-regular Galois extension of $k(T)$ of given group $G$ is $G$-parametric over a given base field $k$ or not indeed seems to be difficult, even for small groups $G$: for example, in the case $G=\mathbb{Z}/3\mathbb{Z}$ and $k=\mathbb{Q}$, the answer seems to be known for only one such extension (this extension is $\mathbb{Z}/3\mathbb{Z}$-parametric over $\mathbb{Q}$; see \S1.1 below). Of course there are some obvious examples like the extensions $k(\sqrt[n]{T})/ k({T})$ ($n \in \mathbb{N} \setminus \{0\}$) and $k(T)(\sqrt{T^2+1})/k(T)$: if $k$ contains the $n$-th roots of unity, the former is $\mathbb{Z}/n\mathbb{Z}$-parametric over $k$ (this follows from the Kummer theory) whereas, if $k \subset \mathbb{R}$, the latter is not $\mathbb{Z}/2\mathbb{Z}$-parametric over $k$ (since none of its specializations is imaginary). But they seem to be quite sparse.

\subsection{Parametric extensions over various fields}
In $\S$2.3, we give some first conclusions on parametric extensions (based on previous works) over various base fields $k$ with good arithmetic properties such as PAC fields, finite fields or the field $\mathbb{Q}$ and its completions. 

For example, in the case $k$ is PAC ($\S$2.3.1), the situation is quite clear: any finite $k$-regular Galois extension of $k(T)$ is parametric over $k$ with respect to any subgroup of its Galois group. In contrast, in the case $k=\mathbb{Q}$ ($\S$2.3.4), not much is known although it may be expected that only a few extensions are parametric. On the one hand, it is known that there is a $G$-parametric extension over $\mathbb{Q}$ of group $G$ for each of the four groups $\{1\}$, $\mathbb{Z}/2\mathbb{Z}$, $\mathbb{Z}/3\mathbb{Z}$ and $S_3$. For any other one, it is unknown whether there exists such an extension or not. On the other hand, only a few non parametric extensions over $\mathbb{Q}$ are known.

\subsection{First examples over $\mathbb{Q}$}
In $\S$3, we use {\it{ad hoc}} arguments to obtain some new examples of non $H$-parametric extensions over $\mathbb{Q}$ with small Galois groups $G$ and small branch point numbers (propositions \ref{ad r=2}, \ref{ad r=3} and \ref{ad r=4}):

\vspace{2mm}

\noindent
{\bf{Theorem 1.}} (1) {\it{A given $\mathbb{Q}$-regular quadratic extension of $\mathbb{Q}(T)$ with two branch points is $\mathbb{Z}/2\mathbb{Z}$-parametric over $\mathbb{Q}$ if and only if each of them is $\mathbb{Q}$-rational.

\vspace{1mm}

\noindent
{\rm{(2)}} No $\mathbb{Q}$-regular Galois extension of $\mathbb{Q}(T)$ of group $\mathbb{Z}/2\mathbb{Z} \times \mathbb{Z}/2\mathbb{Z}$ with three branch points is $\mathbb{Z}/2\mathbb{Z} \times \mathbb{Z}/2\mathbb{Z}$-parametric over $\mathbb{Q}$.

\vspace{1mm}

\noindent
{\rm{(3)}} The splitting extension over $\mathbb{Q}(T)$ of $Y^3+T^2Y+T^2$ is a four branch point $\mathbb{Q}$-regular Galois extension of group $S_3$ which is $H$-parametric over $\mathbb{Q}$ for no subgroup $H \subset S_3$.}}

\vspace{2mm}

\noindent
The proof rests on the non-existence of solutions to some diophantine equations (for the first two parts) and on the non totally real behavior of the specializations (for the third part).

\subsection{A systematic approach}
In $\S$4, we offer a systematic approach, already started in \cite{Leg13a}, to give more examples of non $H$-parametric extensions over $k$ of group $G$ containing $H$. Given a $k$-regular Galois extension $E_{1}/k(T)$ of group $H$ and a $k$-regular Galois extension $E_{2}/k(T)$ of group $G$, we provide two sufficient conditions which each guarantees that there exist some specializations of $E_{1}/k(T)$ of group $H$ which cannot be specializations of $E_{2}/k(T)$ (and so $E_{2}/k(T)$ is not $H$-parametric over $k$). The first one ({\it{Branch Point Hypothesis}}) involves the branch point arithmetic while the second one ({\it{Inertia Hypothesis}}) is a more geometric condition on the inertia of the two extensions $E_{1}/k(T)$ and $E_{2}/k(T)$. Theorem \ref{methode} is our precise result.

We work over base fields $k$ which are quotient fields of any Dedekind domain of characteristic zero with infinitely many distinct primes\footnote{The case there are only finitely many primes can also be considered. We refer to \cite[chapter 3]{Leg13c} where this situation was studied.}, additionaly assumed to be hilbertian. Number fields or finite extensions of rational function fields $\kappa(X)$, with $\kappa$ an arbitrary field of characteristic zero (and $X$ an indeterminate), are typical examples.

\subsection{Applications}
In $\S$5-7, we use our criteria to give new examples of non parametric extensions over various base fields.

\subsubsection{A general result over suitable number fields} In $\S$5, we obtain the following result (corollary \ref{coro non conj}) which leads to non $G$-parametric extensions of group $G$ over suitable number fields for many groups $G$.

\vspace{2mm}

\noindent
{\bf{Theorem 2.}} {\it{Let $G$ be a finite group. Assume that there exists some set $\{C_1,\dots,C_{r},C\}$ of non trivial conjugacy classes of $G$ satisfying the following two conditions:

\vspace{0.5mm}

\noindent
{\rm{(1)}} the elements of $C_1,\dots,C_r$ generate $G$,

\vspace{0.5mm}

\noindent
{\rm{(2)}} the conjugacy class $C$ is a power of $C_i$ for no index $i \in \{1,\dots,r\}$.

\vspace{0.5mm}

\noindent
Then there exist some number field $k$ and some $k$-regular Galois extension of $k(T)$ of group $G$ which is not $G$-parametric over $k$.}}

\vspace{2mm}

\noindent
Many finite groups admit a conjugacy class set as above: abelian groups which are not cyclic of prime power order, symmetric groups $S_n$ ($n \geq 3$), alternating groups $A_n$ ($n \geq 4$), dihedral groups $D_n$ of order $n \geq 2$, non abelian simple groups, {\it{etc.}} See $\S$5.1.1 for more details and references. Moreover the conclusion also holds if $k$ is any finite extension of the rational function field $\mathbb{C}(X)$ (\S5.2) and, under some conjecture of Fried, one can even take $k=\mathbb{Q}$ (corollary \ref{fri}).

\subsubsection{Examples over given base fields} 
In $\S$6 and \S7, we give new examples of non $H$-parametric extensions of group $G$ containing $H$ over various given base fields $k$ (in particular over $k=\mathbb{Q}$). 

To do so, we need to start from two $k$-regular Galois extensions of $k(T)$ with groups $H$ and $G$ respectively. This first step depends on the state-of-the-art in inverse Galois theory, especially in the case $k=\mathbb{Q}$, and the involved finite groups then are the classical ones in this context: abelian groups, symmetric groups, alternating groups, some other simple groups... We present our examples below in connection with those already given in \cite{Leg13a}.

\vspace{3mm}

\noindent
(a) {\it{Examples from the Branch Point Criterion}} (\S6). Let $k$ be a number field and $G$ a finite group. A first example is an improved version of a result of \cite{Leg13a} for $k$-regular Galois extensions of $k(T)$ of group $G$ with four branch points. Here we drop the branch point number assumption and give pure branch point arithmetical conditions for a given $k$-regular Galois extension of $k(T)$ of group $G$ not to be $H$-parametric over $k$ for any given non trivial subgroup $H \subset G$ (corollary \ref{data}).

We give some concrete examples in the situation $G=\mathbb{Z}/2\mathbb{Z}$ (and so $H=\mathbb{Z}/2\mathbb{Z}$ too) where the existence of at least one $k$-regular Galois extension of $k(T)$ of group $G$ satisfying our conditions is guaranteed and which is already of some interest (corollary \ref{Z/2Z}). Some further examples with $G= \mathbb{Z}/n\mathbb{Z}$ ($n \geq 2$) are given (corollaries \ref{Z/2Z 2} and \ref{cyclic 2}).

\vspace{3mm}

\noindent
(b) {\it{Examples from the Inertia Criterion}} (\S7).

\vspace{2.5mm}

\noindent
(i) {\it{Symmetric and alternating groups}}. 
A first example is an improved version of a result of \cite{Leg13a} giving practical sufficient conditions for a given $k$-regular Galois extension of $k(T)$ of group $G=S_n$ ($n \geq 3$) not to be $G=S_n$-parametric over $k$ ($\S$7.1.3); here $k$ is any of our allowed base fields (and even more general ones) while it was a number field in \cite{Leg13a}. We also have an analog with $G=A_n$ ($\S$7.2.3). Theorem 3 below is a consequence of our results:

\vspace{2mm}

\noindent
{\bf{Theorem 3.}} {\it{Let $r$ be an integer $\geq 3$ and $k$ a number field or a finite extension of the rational function field $\mathbb{C}(X)$. Then, for any integer $n \geq 8r^2$, no $k$-regular Galois extension of $k(T)$ of group $G=A_n$ with $r$ branch points is $G=A_n$-parametric over $k$.}}

\vspace{2mm}

\noindent
The same conclusion holds with $G=S_n$ over more general base fields. 

Moreover our results show that several classical $k$-regular Galois extensions of $k(T)$ of group $S_n$ (resp. of group $A_n$) are not $S_n$-parametric (resp. $A_n$-parametric) over any of our allowed base fields $k$. Corollaries \ref{coro Sn} and \ref{coro An} give our main examples.
 
\vspace{2.5mm}

\noindent
(ii) {\it{Non abelian simple groups}}. We also show that some regular realizations of some simple groups $G$ provided by the {\it{rigidity method}} are not $G$-parametric. For instance, using the Atlas \cite{Atl} notation \hbox{for conjugacy classes of finite groups, we have (corollary \ref{PSL}):}

\vspace{2mm}

\noindent
{\it{Let $p$ be a prime $\geq 5$ and $k$ one of our allowed base fields such that $(-1)^{(p-1)/2} p$ is a square in $k$. Then no $k$-regular Galois extension of $k(T)$ of group ${\rm{PSL}}_2(\mathbb{F}_p)$ provided by either one of the rigid triples $(2A, pA, pB)$ (if $(\frac{2}{p})=-1$) and $(3A, pA,pB)$ (if $(\frac{3}{p})=-1$) of conjugacy classes of ${\rm{PSL}}_2(\mathbb{F}_p)$ is ${\rm{PSL}}_2(\mathbb{F}_p)$-parametric over $k$.}}

\vspace{2mm}

\noindent
We also have a similar result with the Monster group (corollary \ref{Monstre}).

\vspace{2.5mm}

\noindent
(iii) {\it{Examples with $H \not=G$}}. We also have various examples which are specifically devoted to the case $H \not=G$. For instance (corollary \ref{baby}):

\vspace{2mm}

\noindent
{\it{Let $k$ be one of our allowed base fields. Then, with {\rm{Th}} the Thompson group, no $k$-regular Galois extension of $k(T)$ of group the Baby-Monster group}} B {\it{provided by the rigid triple $(2C,3A,55A)$ of conjugacy classes of }} B {\it{is {\rm{Th}}-parametric over $k$.}}

\vspace{2mm}

\noindent
Further similar examples with various groups such as symmetric groups, other sporadic groups or $p$-groups are given (corollaries \ref{coro J2} and \ref{cyclic}).

\vspace{2mm}

{\bf{Acknowledgements.}} I am very grateful to my advisor Pierre D\`ebes for his many re-readings, helpful comments and valuable suggestions. I also wish to thank the anonymous referee for constructive remarks.

\section{Basics on parametric extensions}
We first set up in $\S$2.1 the terminology for the basic notions we will use in this paper.  We then point out in $\S$2.2 some connections between parametric extensions and some classical notions in inverse Galois theory and investigate in $\S$2.3 such extensions over various fields.

\subsection{Basic definitions}

Let $k$ be a field and $\overline{k}$ an algebraic closure of $k$. Denote the separable closure of $k$ in $\overline k$ by $k^\sep$ and its absolute Galois group by ${\rm{G}}_k$. Let $E/k(T)$ be a finite Galois extension which is {\it{regular over $k$}} ({\it{i.e.}} $E \cap \overline{k}=k$) and $G$ its Galois group. To make the exposition simpler, say for short that $E/k(T)$ is a ``$k$-regular Galois extension of group $G$". For more on $\S$2.1.1-3, we refer for example to \cite[chapter 3]{Deb09}. 

\subsubsection{Branch points}

Denote the integral closure of $\overline{k}[T]$ (resp. of $\overline{k}[1/T]$) in $E\overline{k}$ by $\overline{B}$ (resp. by $\overline{B^*}$). A point $t_0 \in \overline{k}$ (resp. $\infty$) is said to be {\it{a branch point of $E/k(T)$}} if the prime $(T-t_0) \, \overline{k}[T]$ (resp. $(1/T) \, \overline{k}[1/T]$) ramifies in $\overline{B}$ (resp. in $\overline{B^*}$). Classically $E/k(T)$ has only finitely many branch points, denoted by $t_1,\dots,t_r$.

\subsubsection{Inertia canonical invariant} Assume that $k$ has characteristic zero. Fix a {\it{coherent system $\{\zeta_n\}_{n=1}^\infty$ of roots of unity}}, {\it{i.e.}} $\zeta_n$ is a primitive $n$-th root of unity and $\zeta_{nm}^n=\zeta_m$ for any integers $n$ and $m$.

To each $t_i$ can be associated a conjugacy class $C_i$ of $G$, called the {\it{inertia canonical conjugacy class (associated with $t_i$)}}, in the following way. The inertia groups of $E\overline{k}/\overline{k}(T)$  at $t_i$ are cyclic conjugate groups of order equal to the ramification index $e_i$. Furthermore each of them has a distinguished generator corresponding to the automorphism $(T-t_i)^{1/e_i} \mapsto \zeta_{e_i} (T-t_i)^{1/e_i}$ of $\overline{k}(((T-t_i)^{1/e_i}))$ (replace $T-t_i$ by $1/T$ if $t_i=\infty$). Then $C_i$ is the conjugacy class of all the distinguished generators of the inertia groups at $t_i$. The unordered $r$-tuple $(C_1,\dots,C_r)$ is called {\it{the inertia canonical invariant of $E/k(T)$}}.

\subsubsection{Specializations}

If $t_0 \in \mathbb{P}^1(k)$ is not a branch point, the residue field of some prime above $t_0$ in $E/k(T)$ is denoted by $E_{t_0}$ and we call the extension $E_{t_0}/k$ {\it{the specialization of $E/k(T)$ at $t_0$}} (this does not depend on the choice of the prime above $t_0$ since the extension $E/k(T)$ is Galois). It is a Galois extension of $k$ of Galois group a subgroup of $G$, namely the decomposition group of the extension $E/k(T)$ at $t_0$.

\vspace{2mm}

This classical lemma, which is proved in \cite[\S B.1.4.1]{Leg13c}, is useful:

\begin{lemma} \label{spec}
Let $P(T,Y) \in k[T][Y]$ be a monic (with respect to $Y$) separable polynomial of splitting field $E$ over $k(T)$. Then, for any $t_0 \in k$ such that the specialized polynomial $P(t_0,Y)$ is separable over $k$, $t_0$ is not a branch point of $E/k(T)$ and the specialization $E_{t_0}/k$ of $E/k(T)$ at $t_0$ is the splitting extension over $k$ of $P(t_0,Y)$.
\end{lemma}

\subsubsection{Parametric extensions}

\begin{definition} \label{ext st para}
Let $E/k(T)$ be a $k$-regular finite Galois extension of branch point set $\{t_1, \dots, t_r\}$.

\vspace{1mm}

\noindent
(1) Let $H$ be a subgroup of ${\rm{Gal}}(E/k(T))$. We say that $E/k(T)$ is {\it{H-parametric over $k$}} if, for every Galois extension $F/k$ of group $H$, there exists some point $t_0 \in \mathbb{P}^1(k) \setminus \{t_1,\dots,t_r\}$ such that $F/k$ occurs as the specialization $E_{t_0}/k$ of $E/k(T)$ at $t_0$.

\vspace{1mm}

\noindent
(2) We say that $E/k(T)$ is {\it{parametric over $k$}} if this extension is $H$-parametric over $k$ for each subgroup $H \subset {\rm{Gal}}(E/k(T))$.
\end{definition}

\subsection{Connections with some classical notions}

For more on below and \S2.3, we refer to \cite[\S2.1-2]{Leg13c}.

Let $k$ be a field and $H \subset G$ two finite groups. The notion of $H$-parametric extensions $E/k(T)$ over $k$ of group ${\rm{Gal}}(E/k(T))=G$ is intermediate between these classical notions in inverse Galois theory.

\subsubsection{One parameter generic polynomials}
Recall that a monic (with respect to $Y$) separable polynomial $P(T,Y) \in k[T][Y]$ of group $G$ is called {\it{generic over $k$}} if, for any field extension $L/k$, any Galois extension of $L$ of group $G$ occurs as the splitting extension of some specialized separable polynomial $P(t_0,Y)$ with $t_0 \in L$. See \cite{JLY02} for more on generic polynomials. 

Given a generic polynomial over $k$ of group $G$, its splitting extension $E/k(T)$ is $G$-parametric over $k$ (lemma \ref{spec}). It is in fact {\it{$G$-generic over $k$}}, {\it{i.e.}} the extension $EL/L(T)$ is $G$-parametric over $L$ for any field extension $L/k$ (which is linearly disjoint from $E$ over $k$). Of course any $G$-generic extension over $k$ of group $G$ is $G$-parametric over $k$. Remark \ref{non gen} shows however that the converse does not hold in general\footnote{Other counter-examples are given in \cite[remark 2.1.7 and example 2.2.1]{Leg13c}.}.

Moreover, if $k$ is infinite, the splitting extension $E/k(T)$ is even {\it{generic over $k$}}, {\it{i.e.}} the extension $EL/L(T)$ is parametric over $L$ for any field extension $L/k$ (as explained in \cite[proposition 2.1.8]{Leg13c}, it is essentially \cite{Kem01}). In particular, $E/k(T)$ is parametric over $k$.

\subsubsection{Lifting extensions}
Given a Galois extension $F/k$ of group $H$, recall that a {\it{lifting extension of group $G$ for $F/k$}} is a $k$-regular Galois extension $E_F/k(T)$ of group $G$ which has the extension $F/k$ among its specializations. 

Then any $H$-parametric extension over $k$ of group $G$ obviously is a lifting extension of group $G$ for any Galois extension of $k$ of group $H$. Moreover, if there exists at least one $G$-parametric extension over $k$ of group $G$, then it obviously solves the {\it{Beckmann-Black problem for $G$ over $k$}}, which asks whether any Galois extension of $k$ of group $G$ has a lifting extension with the same group.

\subsection{Parametric extensions over various fields} Let $H \subset G$ be two finite groups. We investigate below $H$-parametric extensions of group $G$ over various base fields $k$.

\subsubsection{$k$ is a PAC field}
Recall that a field $k$ is said to be {\it{PAC}} if every non-empty geometrically irreducible $k$-variety has a Zariski-dense set of $k$-rational points. Classical results show that in some sense PAC fields are ``abundant" \cite[theorem 18.6.1]{FJ05} and a concrete example (due to Pop) is the field $\mathbb{Q}^{\rm{tr}}(\sqrt{-1})$; here $\mathbb{Q}^{\rm{tr}}$ denotes the field of totally real numbers (algebraic numbers such that all conjugates are real). See \cite{FJ05} for more on PAC fields.

In the case $k$ is a PAC field, the situation is quite clear: \cite[theorem 3.2]{Deb99a} shows that any $k$-regular Galois extension of $k(T)$ of group $G$ (such an extension exists \cite{FV91} \cite{Pop96}) is parametric over $k$.

\subsubsection{$k$ is a finite field}
Since there are no (resp. only one) Galois extension of $k$ of group $H$ if $H$ is not cyclic (resp. if $H$ is cyclic), we trivially have that any $k$-regular Galois extension of $k(T)$ of group $G$ is $H$-parametric over $k$ if $H$ is not cyclic and $H'$-parametric over $k$ for at least one cyclic subgroup $H' \subset G$.

Moreover any $k$-regular Galois extension of $k(T)$ of group $G$ is known to be parametric over $k$ provided that $k$ is large enough (depending on $G$ and the branch point number) \cite{Fri74} \cite{Jar82} \cite{Eke90} \cite{DG11}. As in addition the group $G$ occurs as the Galois group of a $k$-regular Galois extension of $k(T)$ provided that $k$ is large enough (depending on $G$) \cite{FV91} \cite{Pop96} (see \cite[remark 3.9(a)]{DD97b} for more details), conclude that there exists at least one parametric extension over $k$ of group $G$ for large enough finite fields $k$.

\subsubsection{$k$ is a completion of $\mathbb{Q}$}

\noindent
\vspace{2.25mm}

\noindent
(a) $k=\mathbb{Q}_p$. Since any finite Galois extension of $\mathbb{Q}_p$ is solvable, we vacuously have that any $\mathbb{Q}_p$-regular Galois extension of $\mathbb{Q}_p(T)$ of group $G$ (such an extension exists \cite{Har87}) is $H$-parametric if $H$ is not solvable.

If $H$ is solvable, it does not hold in general. Indeed, given a $\mathbb{Q}$-regular Galois extension $E/\mathbb{Q}(T)$ of group $\mathbb{Z}/8\mathbb{Z}$, the extension $E\mathbb{Q}_2/\mathbb{Q}_2(T)$ is not $\mathbb{Z}/8\mathbb{Z}$-parametric over $\mathbb{Q}_2$. Otherwise there exists some specialization point $t_0 \in \mathbb{P}^1(\mathbb{Q}_2)$ such that $(E\mathbb{Q}_2)_{t_0}/\mathbb{Q}_2$ is the unramified extension of $\mathbb{Q}_2$ of degree $8$. From Krasner's lemma, one may assume that $t_0 \in \mathbb{P}^1(\mathbb{Q})$ and one then obtains a contradiction from \cite{Wan48}.

\vspace{2.25mm}

\noindent
(b) $k=\mathbb{R}$. Since the only finite extensions of $\mathbb{R}$ are the trivial one $\mathbb{R}/\mathbb{R}$ and the quadratic one $\mathbb{C}/\mathbb{R}$, we trivially have that any $\mathbb{R}$-regular Galois extension of $\mathbb{R}(T)$ of group $G$ (such an extension is known to be from a classical work of Hurwitz) is $H$-parametric over $\mathbb{R}$ if neither $H = \{1\}$ nor $H=\mathbb{Z}/2\mathbb{Z}$, and is $\{1\}$-parametric or $\mathbb{Z}/2\mathbb{Z}$-parametric over $\mathbb{R}$. In particular, any $\mathbb{R}$-regular Galois extension of $\mathbb{R}(T)$ of group $G$ is parametric over $\mathbb{R}$ if $G$ has odd order.

If $G$ has even order, there is at least one $\mathbb{Z}/2\mathbb{Z}$-parametric extension over $\mathbb{R}$ with group $G$ \footnote{The same claim also holds with $\mathbb{Z}/2\mathbb{Z}$ replaced by $\{1\}$ (see {\it{e.g.}} the proof of \cite[theorem 4.3.2]{Deb09}).}. Indeed, from \cite{Hur91} \cite {KN71}, it suffices to find an even integer $r\geq 2$ and an $(r+1)$-tuple $(g_0,g_1,\dots,g_r)$ of non trivial elements of $G$ such that

\noindent
- $g_1 \dots g_r = 1$ and $\langle g_1, \dots, g_r \rangle = G$,

\noindent
- $g_0$ has order 2,

\noindent
- $g_{r+1-i} = g_0 g_i^{-1} g_0$ for each index $i \in \{1,\dots,r/2\}$ \footnote{and the desired regular realization of $G$ over $\mathbb{R}$ even has no real branch point.}.

To do this, start from an element $g_1 \in G$ that has order 2. Next pick non trivial elements $g_2, \dots, g_r$ of $G$ such that $g_1, \dots, g_r$ generate $G$. Denote the elements $g_{r}^{-1}, \dots, g_{2}^{-1}$ by $g_{r+1}, \dots,  g_{2r-1}$ and set $g_0=g_1$. Then the following $(4r-1)$-tuple satisfies the desired conditions:
$$(g_0, g_0 g_{2r-1}^{-1} g_0, \dots, g_0 g_{1}^{-1} g_0, g_1, \dots, g_{2r-1})$$

\subsubsection{$k=\mathbb{Q}$} The situation in the case $k=\mathbb{Q}$ is more unclear.

\vspace{2.25mm}

\noindent
(a) If $G = \{1\}, \mathbb{Z}/2\mathbb{Z}, \mathbb{Z}/3\mathbb{Z}$ or $S_3$, there is a parametric extension over $\mathbb{Q}$ of group $G$. This comes from the fact that these four groups (are the only ones to) have a one parameter generic polynomial over $\mathbb{Q}$ \cite[page 194]{JLY02}. Below are some examples of parametric extensions over $\mathbb{Q}$ which are in fact provided by one parameter generic polynomials:

\vspace{1mm}

\noindent
(i) the trivial extension $\mathbb{Q}(T)/\mathbb{Q}(T)$,

\vspace{0.5mm}

\noindent
(ii) the $\mathbb{Q}$-regular quadratic extension $\mathbb{Q}(\sqrt{T})/\mathbb{Q}(T)$,

\vspace{0.5mm}

\noindent
(iii) the $\mathbb{Q}$-regular cyclic extension of $\mathbb{Q}(T)$ of degree 3 defined by the polynomial $Y^3-TY^2+(T-3)Y+1$ ({\it{e.g.}} \cite[\S2.1]{JLY02}),

\vspace{0.5mm}

\noindent
(iv) the $\mathbb{Q}$-regular Galois extension of $\mathbb{Q}(T)$ of group $S_3$ defined by the trinomial $Y^3+TY+T$ ({\it{e.g.}} \cite[\S2.1]{JLY02}).

\vspace{2.25mm}

\noindent
(b) If $G$ is none of these four groups, it is unknown whether there is a $G$-parametric extension over $\mathbb{Q}$ of group $G$ or not. In the case $H \not=G$, \cite[proposition 3.2.4]{Deb09} provides an $H$-parametric extension over $\mathbb{Q}$ of group $G$ in the case $H=\{1\}$ and $G$ abelian.

\vspace{2.25mm}

\noindent
(c) In addition to the example with $G=\mathbb{Z}/2\mathbb{Z}$ from the presentation, only a few negative examples are known.

\vspace{1mm}

\noindent
(i) No $\mathbb{Q}$-regular Galois extension of $\mathbb{Q}(T)$ of group $S_7$ and branch point set $\{0, 1, \infty\}$ is $S_7$-parametric over $\mathbb{Q}$: \cite[example 1.1]{Bec94} shows indeed that the Galois extension of $\mathbb{Q}$ of group $S_7$ defined by the polynomial $P(Y)= Y^7 + 42482 Y^6 + 5643 Y^5 -21164 Y^4 + 2431 Y^3 + 46189 Y^2 + 46189 Y + 46189$ cannot be a specialization of such an extension.

\vspace{1mm}

\noindent
(ii) For any finite group $G \not= \mathbb{Z}/2\mathbb{Z} \times \mathbb{Z}/2\mathbb{Z}$, $S_3$, $D_4$, $D_6$ which occurs as the Galois group of a totally real Galois extension of $\mathbb{Q}$, no $\mathbb{Q}$-regular Galois extension of $\mathbb{Q}(T)$ of group $G$ with three branch points is $G$-parametric over $\mathbb{Q}$: \cite[proposition 1.2]{DF90} shows indeed that no specialization of such an extension is totally real.

\section{First examples} This section is devoted to theorem 1 from the presentation. We use {\it{ad hoc}} arguments to give new examples of non $H$-parametric extensions over $\mathbb{Q}$ of group $G$. Our examples have $r \in \{2,3,4\}$ branch points ($\S$3.1-3). We also discuss the case $r \geq 5$ in $\S$3.4.

\subsection{An example with $r=2$}\footnote{More on $\mathbb{Q}$-regular finite Galois extensions of $\mathbb{Q}(T)$ with $r=2$ branch points can be found in \cite[\S2.3.2]{Leg13c}.} Proposition \ref{ad r=2} below unifies the two examples $\mathbb{Q}(T)(\sqrt{T^2+1})/\mathbb{Q}(T)$ and $\mathbb{Q}(\sqrt{T})/\mathbb{Q}(T)$:

\begin{proposition} \label{ad r=2}
Let $a$, $b$ and $c$ three rational numbers such that $b^2-4ac \not=0$ and $E=\mathbb{Q}(T)(\sqrt{aT^2+bT+c})$. Then the following three conditions are equivalent:

\vspace{0.5mm}

\noindent
{\rm{(1)}} $E/\mathbb{Q}(T)$ is parametric over $\mathbb{Q}$,

\vspace{0.5mm}

\noindent
{\rm{(2)}} $E/\mathbb{Q}(T)$ is $\mathbb{Z}/2\mathbb{Z}$-parametric over $\mathbb{Q}$,

\vspace{0.5mm}

\noindent
{\rm{(3)}} $b^2-4ac$ is a square in $\mathbb{Q}$.
\end{proposition}

Note that condition (3) is equivalent to the following:

\vspace{1.3mm}

\noindent
(4) {\it{the two branch points of $E/\mathbb{Q}(T)$ each is $\mathbb{Q}$-rational.}}

\vspace{1.3mm}

Indeed this easily follows from lemma \ref{grolem} below which will be used on several occasions in this paper. We omit the proof which involves very classical tools and which is detailed in \cite[\S2.3.2.1]{Leg13c}.

\begin{lemma} \label{grolem}
Let $k$ be a field of characteristic zero, $P(T) \in k[T]$ a separable polynomial over $k$, $n$ its degree and $\{t_1,\dots,t_n\}$ its root set. Then the quadratic extension $k(T)(\sqrt{P(T)})/k(T)$ is $k$-regular and its branch point set {\bf{t}} is

\noindent
{\rm{(1)}} ${\bf{t}} = \{t_1,\dots,t_n\}$ if $n$ is even,

\vspace{0.5mm}

\noindent
{\rm{(2)}} ${\bf{t}} = \{t_1,\dots,t_n\} \cup \{\infty \}$ if $n$ is odd.
\end{lemma}

\begin{proof}[Proof of proposition \ref{ad r=2}]
We successively prove implications (3) $\Rightarrow$ (1), (1) $\Rightarrow$ (2) and (2) $\Rightarrow$ (3). Furthermore the proof will show the following:

\vspace{1.5mm}

\noindent
{\rm{(a)}} {\it{if condition (3) holds, then any quadratic or trivial extension of $\mathbb{Q}$ is the splitting extension over $\mathbb{Q}$ of some specialized polynomial $Y^2-(at_0^2 + bt_0 + c)$ with $t_0 \in \mathbb{Q}$}},

\vspace{1.5mm}

\noindent
{\rm{(b)}} {\it{if condition (3) does not hold, then there are infinitely many distinct quadratic extensions of $\mathbb{Q}$ which each is not a specialization of $E/\mathbb{Q}(T)$.}}

\vspace{2.25mm}

\noindent
{\it{(3) $\Rightarrow$ (1)}}. Assume that condition (3) holds. Let $t_1 \in \mathbb{Q}$ be a root of $aT^2+bT+c$ and $F/\mathbb{Q}$ a quadratic or trivial extension. Set $F=\mathbb{Q}(\sqrt{d})$ with $d$ a non-zero integer. 

The curve defined by the equation $dY^2 = aT^2+bT+c$ has a (non singular) $\mathbb{Q}$-rational point (for example $(0,t_1)$). Being of genus 0, it is then birational to $\mathbb{P}^1$ over $\mathbb{Q}$. Then there exist two rational numbers $y$ and $t_0$ such that $y \not=0$ and $dy^2 = at_0^2+bt_0+c$. Hence one has $F=\mathbb{Q}(\sqrt{at_0^2+bt_0+c})$, {\it{i.e.}} $F/\mathbb{Q}$ is the splitting extension over $\mathbb{Q}$ of the specialized polynomial $Y^2-(at_0^2+bt_0+c)$ (and so statement (a) holds). Since this polynomial is separable over $\mathbb{Q}$, one may apply lemma \ref{spec} and conclude that $F/\mathbb{Q}$ is the specialization $E_{t_0}/\mathbb{Q}$.

\vspace{2mm}

\noindent
{\it{(1) $\Rightarrow$ (2)}}. This is a consequence of definition \ref{ext st para}.

\vspace{2mm}

\noindent
{\it{(2) $\Rightarrow$ (3)}}. Assume that condition (2) holds. There are three steps to show that $b^2-4ac$ is a square in $\mathbb{Q}$.

\vspace{1.5mm}

\noindent
- {\it{Step 1}}: $a \in \mathbb{Z} \setminus \{0\}$, $b=0$ and $c \in \mathbb{Z} \setminus \{0\}$. First remark that $\infty$ is not a branch point since $a \not=0$ (lemma \ref{grolem}).

Let $p$ be a prime such that neither $a$ nor $c$ is a multiple of $p$ and that does not ramify in $E_\infty/\mathbb{Q}$. From condition (2), there exists some $t_0 \in \mathbb{Q}$ such that $E_{t_0} = \mathbb{Q}(\sqrt{p})$, {\it{i.e.}} $\mathbb{Q}(\sqrt{at_0^2+c}) = \mathbb{Q}(\sqrt{p})$ (lemmas \ref{spec} and \ref{grolem}). Hence there exists some non-zero rational number $\lambda$ such that $p\lambda^2=at_0^2+c$. Then there exist three non-zero integers $x$, $y$ and $z$ such that $px^2=ay^2+cz^2$ and one may assume that $z$ is not a multiple of $p$ (otherwise $y$ and $x$ are also multiples of $p$ and, with $n$ the $p$-adic valuation of $z$, one may then replace $(x,y,z)$ by $(x/p^n, y/p^n, z/p^n)$). By reducing modulo $p$, $-ac$ is a square modulo $p$.

Hence $Y^2 + 4ac$ has a root modulo $p$ for all but finitely many primes $p$ (note that this also holds if we only assume that all but finitely many quadratic extensions of $\mathbb{Q}$ are specializations of $E/\mathbb{Q}(T)$, so proving statement (b)). From {\it{e.g.}} \cite[theorem 9]{Hei67}\footnote{It seems that more elementary proofs exist in the quadratic case.}, $-4ac$ is a square in $\mathbb{Q}$.

\vspace{1.5mm}

\noindent
- {\it{Step 2: $(a,b,c) \in \mathbb{Z}^3$}}. Condition (3) trivially holds if $a=0$ or $c=0$. So assume that $a \not=0$ and $c \not=0$. Set $\Delta=b^2-4ac$. 

Let $p$ be a prime such that neither $a$ nor $\Delta$ is a multiple of $p$ and that does not ramifiy in $E_\infty /\mathbb{Q}$. From condition (2), there is some $t_0 \in \mathbb{Q}$ such that $\mathbb{Q}(\sqrt{p})/\mathbb{Q}=E_{t_0}/\mathbb{Q}$, {\it{i.e.}} $\mathbb{Q}(\sqrt{p})=\mathbb{Q}(\sqrt{a t_0^2+b t_0+c})$. Set ${t}_0'=2a t_0 +b$. As $a {t}_0'^2 -a \Delta=4a^2(at_0^2+bt_0+c)$, one has $\mathbb{Q}(\sqrt{a t_0^2+bt_0+c})=\mathbb{Q}(\sqrt{a {t}_0'^2 -a \Delta})$. From step 1, $4a^2 \Delta$ is a square in $\mathbb{Q}$ and so is $\Delta$ too.

\vspace{1.5mm}

\noindent
- {\it{Step 3: $(a,b,c) \in \mathbb{Q}^3$}}. Set $a=a_1/a_2$, $b=b_1/b_2$ and $c=c_1/c_2$ with integers $a_1, a_2, b_1,b_2,c_1,c_2$ such that $(a_1,a_2)=(b_1,b_2)=(c_1,c_2)=1$.

As $a_2^2  b_2^2  c_2^2  (a T^2+b T+c)= a_1  a_2  b_2^2  c_2^2  T^2 + b_1  b_2  a_2^2  c_2^2  T+ c_1  c_2  a_2^2  b_2^2$, one has $$E=\mathbb{Q}(T)(\sqrt{a_1 \, a_2 \, b_2^2 \, c_2^2 \, T^2 + b_1 \, b_2 \, a_2^2 \, c_2^2 \, T+ \,c_1 \, c_2 \, a_2^2 \, b_2^2})$$ From step 2, the discriminant $b_1^2 \, b_2^2 \, a_2^4 \, c_2^4 - 4 \, a_1 \, c_1 \, a_2^3 \, c_2^3 \, b_2^4 = (a_2 \, b_2 \, c_2)^4 (b^2-4\, ac)$ is a square in $\mathbb{Q}$. Hence condition (3) holds.
\end{proof}

\subsection{An example with $r=3$}
As for any abelian finite group, the Beckmann-Black problem for $\mathbb{Z}/2\mathbb{Z} \times \mathbb{Z}/2\mathbb{Z}$ over $\mathbb{Q}$ has a positive answer: any Galois extension $F/\mathbb{Q}$ of group $\mathbb{Z}/2\mathbb{Z} \times \mathbb{Z}/2\mathbb{Z}$ has a lifting extension $E_F/\mathbb{Q}(T)$ with the same group. 
Moreover \cite[corollary 2.4]{Bec94} shows that $E_F/\mathbb{Q}(T)$ may be chosen with three branch points.

The following shows however that none of these lifting extensions $E_F/\mathbb{Q}(T)$ with three branch points is $\mathbb{Z}/2\mathbb{Z} \times \mathbb{Z}/2\mathbb{Z}$-parametric over $\mathbb{Q}$:

\begin{proposition} \label{ad r=3}
Let $E/\mathbb{Q}(T)$ be a $\mathbb{Q}$-regular Galois extension of group $\mathbb{Z}/2\mathbb{Z} \times \mathbb{Z}/2\mathbb{Z}$ with three branch points. Then there exist infinitely many distinct Galois extensions of $\mathbb{Q}$ of group $\mathbb{Z}/2\mathbb{Z} \times \mathbb{Z}/2\mathbb{Z}$ which each is not a specialization of $E/\mathbb{Q}(T)$. In particular, this extension is not $\mathbb{Z}/2\mathbb{Z} \times \mathbb{Z}/2\mathbb{Z}$-parametric over $\mathbb{Q}$.
\end{proposition}

\begin{proof}
Let $P_1(T)$ and $P_2(T)$ be two distinct separable polynomials over $\mathbb{Q}$ such that $E=\mathbb{Q}(T) (\sqrt{P_1(T)}, \sqrt{P_2(T)})$. 

Given $i \in \{1,2\}$, it follows from the extension $E/\mathbb{Q}(T)$ having three branch points and the $\mathbb{Q}$-regular quadratic one $\mathbb{Q}(T) (\sqrt{P_i(T)})/\mathbb{Q}(T)$ having an even branch point number (lemma \ref{grolem}) that the latter has two branch points. Consequently each branch point of $E/\mathbb{Q}(T)$ is $\mathbb{Q}$-rational. Hence we may assume that these branch points are 0, 1 and $\infty$. In particular, there exist two non-zero squarefree integers $a$ and $b$ such that $E=\mathbb{Q}(T)(\sqrt{a\, T},\sqrt{b\, T-b})$ (lemma \ref{grolem}).

Now fix two distinct squarefree integers $d_1$, $d_2$ and assume that $\mathbb{Q}(\sqrt{d_1},\sqrt{d_2}) = \mathbb{Q}(\sqrt{a\, t_0}, \sqrt{b\,t_0-b})$ for some $t_0 \in \mathbb{Q} \setminus \{0,1\}$. Then the quadratic subextensions coincide and one of these conditions holds:

\noindent
(i) $a\,d_1\,t_0 \in \mathbb{Q}^2$ and $d_2 \,(b\,t_0-b) \in \mathbb{Q}^2$,

\noindent
(ii) $a\,d_1\,t_0 \in \mathbb{Q}^2$ and $d_1\,d_2 \,(b\,t_0-b) \in \mathbb{Q}^2$,

\noindent
(iii) $a\,d_2\,t_0 \in \mathbb{Q}^2$ and $d_1 \,(b\,t_0-b) \in \mathbb{Q}^2$,

\noindent
(iv) $a\,d_2\,t_0 \in \mathbb{Q}^2$ and $d_1 \, d_2 \,(b\,t_0-b) \in \mathbb{Q}^2$,

\noindent
(v) $a\,d_1 \, d_2\,t_0 \in \mathbb{Q}^2$ and $d_1  \,(b\,t_0-b) \in \mathbb{Q}^2$,

\noindent
(vi) $a\, d_1 \, d_2\,t_0 \in \mathbb{Q}^2$ and $d_2 \,(b\,t_0-b) \in \mathbb{Q}^2$.

\noindent
Consequently one of the following six equations has a non trivial solution, {\it{i.e.}} a solution $(x,y,z) \in \mathbb{Z}^3$ such that $xyz \not=0$:

\noindent
(i) $a\,d_1\, X^2 - b\,d_2\, Y^2 - Z^2=0$,

\noindent
(ii) $a \,X^2 - b\,d_2 \, Y^2 - d_1\,Z^2=0$,

\noindent
(iii) $a\,d_2\, X^2 - b \,d_1 \, Y^2 - Z^2=0$,

\noindent
(iv) $a\, X^2 - b \, d_1 Y^2 - d_2 \, Z^2=0$,

\noindent
(v) $a \,  d_2 X^2  -b \, Y^2 - d_1 \, Z^2=0$,

\noindent
(vi) $a \, d_1 \, X^2 - b \, Y^2 - d_2 \, Z^2=0$.

\noindent
We show below that there are infinitely many distinct couples $(d_1,d_2)$ of distinct squarefree integers such that none of these six equations has a non trivial solution. In particular, the conclusion holds (lemma \ref{spec}).

One may assume that $a>0$ or $b<0$ (otherwise take $d_1 >0$ and $d_2 >0$ to conclude). Assume for example that $a>0$ and $b >0$ (the other two cases for which $a >0$ and $b <0$, $a<0$ and $b<0$ are similar). 

First assume that the squarefree integer $b$ satisfies $b \not=1$. Fix a squarefree integer $d_2>0$ such that neither $a\,b\, d_2$ nor $a \, d_2$ is a square in $\mathbb{Q}$. As $b \not =1$, the quadratic fields $\mathbb{Q}(\sqrt{a\,b\, d_2})$ and $\mathbb{Q}(\sqrt{a \, d_2})$ are distinct. Hence there are infinitely many distinct primes $p$ such that neither $a\,b\, d_2$ nor $a \, d_2$ is a square modulo $p$ ({\it{e.g.}} \cite[theorem 7]{Nag69}). Then, for such a prime $p$, none of the previous equations with $d_1=-p$ has a non trivial solution, {\it{i.e.}} none of the following equations does:

\noindent
(i) $-a\, p \, X^2 - b\, d_2 \, Y^2 -   Z^2=0$,

\noindent
(ii) $a\,X^2 - b \, d_2 \, Y^2 + p\,Z^2=0$,

\noindent
(iii) $a \, d_2 \, X^2 + p \, b\, Y^2 -  Z^2=0$,

\noindent
(iv) $a \,X^2 + p \, b \, Y^2 -  d_2 \,Z^2=0$,

\noindent
(v) $a\, d_2 \, X^2 - b \,Y^2 + p  \,Z^2=0$,

\noindent
(vi) $-a \,  p \,  X^2 - b \,Y^2 - d_2 \,Z^2=0$.

\noindent
Indeed, first note that neither equation (i) nor equation (vi) has such a solution (as all coefficients are negative). If one of equations (ii)-(v) has such a solution $(x,y,z)$, one may assume that $x$ is not a multiple of $p$ (otherwise $y$ and $z$ are also multiples of $p$ and, with $n$ the $p$-adic valuation of $x$, one may then replace $(x,y,z)$ by $(x/p^n, y/p^n, z/p^n)$). By reducing modulo $p$, $a \, d_2$ or $a \, b \, d_2$ is a square modulo $p$; a contradiction.

Now assume that $b=1$. Fix a squarefree integer $d_2 >0$ such that $a \, d_2$ is not a square in $\mathbb{Q}$. Hence there exist infinitely many distinct primes $p$ such that $a \, d_2$ is not a square modulo $p$ ({\it{e.g.}} \cite[theorem 9]{Hei67}). Then, for such a prime $p$, a similar argument as that in the case $b \not=1$ shows that none of the previous six equations with $d_1=-p$ has a non trivial solution, thus ending the proof.
\end{proof}

\begin{remark}
The proof and proposition \ref{ad r=2} show in particular that any quadratic subextension of $E/\mathbb{Q}(T)$ is parametric over $\mathbb{Q}$. However their compositum is not $\mathbb{Z}/2\mathbb{Z} \times \mathbb{Z}/2\mathbb{Z}$-parametric over $\mathbb{Q}$.
\end{remark}

\subsection{An example with $r=4$} Denote the splitting extension over $\mathbb{Q}(T)$ of $Y^3+T^2Y+T^2$ by $E/\mathbb{Q}(T)$. As this trinomial is absolutely irreducible and its discriminant $\Delta(T)=-4T^6-27T^4$ is not a square in $\overline{\mathbb{Q}}(T)$, $E/\mathbb{Q}(T)$ is $\mathbb{Q}$-regular and one has ${\rm{Gal}}(E/\mathbb{Q}(T))=S_3$.

\begin{proposition} \label{ad r=4}
The extension $E/\mathbb{Q}(T)$ is $H$-parametric over $\mathbb{Q}$ for no subgroup $H \subset S_3$. More precisely, for any non trivial subgroup $H \subset S_3$, there exist infinitely many distinct Galois extensions of $\mathbb{Q}$ of group $H$ which each is not a specialization of $E/\mathbb{Q}(T)$.
\end{proposition}

\begin{proof}
One easily shows that the branch point set {\bf{t}} of $E/\mathbb{Q}(T)$ is contained in $\{0,3i\sqrt{3}/2,-3i\sqrt{3}/2, \infty\}$. Hence {\bf{t}} contains some $\mathbb{Q}$-rational point and the two complex conjugate points $3i\sqrt{3}/2$, $-3i \sqrt{3}/2$.

Assume that $E/\mathbb{Q}(T)$ has three branch points. The Riemann-Hurwitz formula then shows that it has genus 0. Then there exists some transcendental element $U$ over $\mathbb{Q}$ such that $E\overline{\mathbb{Q}} = \overline{\mathbb{Q}}(U)$. Since $S_3$ is isomorphic to the finite group $\mathcal{D}$ generated by $\sigma$ and $\tau$ such that $\tau(U)=1/U$ and $\sigma(U)=e^{2i \pi/3}U$, one has $\overline{\mathbb{Q}}(T)=\overline{\mathbb{Q}}(U)^\mathcal{D}=\overline{\mathbb{Q}}(U^3+U^{-3})$ (since $U$ is a root ot the trinomial $Y^6-(U^3+U^{-3}) Y^3 + 1$). Moreover the branch point set of $\overline{\mathbb{Q}}(U)/\overline{\mathbb{Q}}(U^3+U^{-3})$ is contained in $\{-2,2,\infty\}$. In particular, any branch point of $E\overline{\mathbb{Q}}(T)/\overline{\mathbb{Q}}(T)$ should be $\mathbb{Q}$-rational; a contradiction. Hence the extension $E/\mathbb{Q}(T)$ has four branch points.

Given a non-zero rational number $t_0$, the specialized polynomial $Y^3+t_0^2Y+t_0^2$ is separable over $\mathbb{Q}$ and, from lemma \ref{spec}, the specialization $E_{t_0}/\mathbb{Q}$ is its splitting extension over $\mathbb{Q}$. Since this polynomial has only one real root, the specialization $E_{t_0}/\mathbb{Q}$ is not totally real. Hence the conclusion obviously holds for $H= \{1\}$ and $H=\mathbb{Z}/2\mathbb{Z}$. Moreover, since any finite Galois extension of $\mathbb{Q}$ of odd degree is totally real, the conclusion also holds for $H=\mathbb{Z}/3\mathbb{Z}$ \footnote{In fact no Galois extension of $\mathbb{Q}$ of group $\mathbb{Z}/3\mathbb{Z}$ is a specialization of $E/\mathbb{Q}(T)$.}. Finally, since it is known that there exist infinitely many distinct totally real Galois extensions of $\mathbb{Q}$ of group $S_3$ ({\it{e.g.}} \cite[proposition 2]{KM01}), the conclusion is also true for $H=S_3$, thus ending the proof.
\end{proof}

\subsection{The case $r \geq 5$} In this case, it seems difficult to give similar examples. However one has the following general conclusion:

\vspace{2mm}

\noindent
{\it{Let $G$ be a finite group, $H$ a subgroup of $G$ and $E/\mathbb{Q}(T)$ a $\mathbb{Q}$-regular Galois extension of group $G$ with $r \geq 5$ branch points. Then, given a Galois extension $F/\mathbb{Q}$ of group $H$, there exist only finitely many distinct points $t_0 \in \mathbb{P}^1(\mathbb{Q})$ (possibly none) such that the extension $F/\mathbb{Q}$ occurs as the specialization $E_{t_0}/\mathbb{Q}$ of $E/\mathbb{Q}(T)$ at $t_0$.}}

\vspace{2mm}

\noindent
Indeed denote the genus of $E/\mathbb{Q}(T)$ by g. The Riemann-Hurwitz for-\hbox{mula yields $2{\rm{g}} \geq 2 + [E  :  \mathbb{Q}(T)] ((r/2)-2)$. Hence ${\rm{g}} \geq 2$ and the con-} 

\noindent
{clusion} follows from the {\it{Faltings theorem}} as explained in \cite[$\S$3.3.5]{Deb99a}.

\section{Criteria for non parametricity}

This section is devoted to theorem \ref{methode} below which gives our most general criteria for a given $k$-regular Galois extension of $k(T)$ not to be parametric over $k$; it is the aim of $\S$4.1. We also give in $\S$4.2 several more practical forms of this statement which each will be used in the next three sections to obtain new examples of such extensions over various base fields $k$.

Let $A$ be a Dedekind domain of characteristic zero with infinitely many distinct primes and $k$ its quotient field. First recall the following:

\begin{definition}
Let $P(T) \in k[T]$ be a non constant polynomial and $\mathcal{P}$ a (non-zero) prime of $A$. We say that {\it{$\mathcal{P}$ is a prime divisor of $P(T)$}} if there exists some $t_0 \in k$ such that $P(t_0)$ is in the maximal ideal $\mathcal{P}A_\mathcal{P}$ of the localization $A_\mathcal{P}$ of $A$ at $\mathcal{P}$.
\end{definition}

\subsection{General result}

\subsubsection{Notation} Let $H$ be a non trivial finite group and $E_1/k(T)$ a $k$-regular Galois extension of group $H$. Denote its  branch point set by $\{t_{1,1},\dots,t_{r_1,1}\}$ and its inertia canonical invariant by $(C_{1,1},\dots,C_{r_1,1})$.

Recall some important notation from \cite{Leg13a}. For each index $i \in \{1,\dots,r_1\}$, denote the irreducible polynomial of $t_{i,1}$ (resp. of $1/t_{i,1}$ \footnote{Set $1/t_{i,1}=0$ if $t_{i,1} = \infty$ and $1/t_{i,1}=\infty$ if $t_{i,1}=0$.}) over $k$ by $m_{i,1}(T)$ (resp. by $m_{i,1}^*(T)$). Set $m_{i,1}(T) = 1$ if $t_{i,1}=\infty$ and $m_{i,1}^*(T)=1$ if $t_{i,1}=0$. Finally set $m_{E_1}(T) = \prod_{i=1}^{r_1} m_{i,1}(T)$ and $m_{E_1}^*(T) = \prod_{i=1}^{r_1} m_{i,1}^*(T)$.

Let $G$ be a finite group containing $H$ and $E_2/k(T)$ a $k$-regular Galois extension of group $G$. Define the same notation for $E_2/k(T)$. Moreover, given a conjugacy class $C$ of $H$, denote the conjugacy class in $G$ of elements of $C$ by $C^G$.

\subsubsection{Statement of the result} Consider the following two conditions:

\vspace{1.5mm}

\noindent
(Branch Point Hypothesis) {\it{there exist infinitely many distinct primes of $A$ which each is a prime divisor of $m_{E_1}(T)  \cdot m_{E_1}^*(T)$ but not of $m_{E_2}(T) \cdot m_{E_2}^*(T)$}},

\vspace{1.5mm}

\noindent
(Inertia Hypothesis) {\it{there exists some index $i \in \{1,\dots,r_1\}$ satisfying the following two conditions:

\vspace{0.5mm}

{\rm{(a)}} $m_{i,1}(T) \cdot m^*_{i,1}(T)$ has infinitely many distinct prime divisors,

\vspace{0.5mm}

{\rm{(b)}} the set $\{C_{1,2}^{a}, \dots, C_{r_2,2}^{a} \,  /  \, a \in \mathbb{N}\}$ does not contain $C_{i,1}^G$}}.

\begin{theorem} \label{methode}
Under either one of these two conditions, the following non parametricity condition holds:

\vspace{1mm}

\noindent
{\rm{(non parametricity)}} there exist infinitely many distinct finite Galois extensions of $k$ which each is not a specialization of $E_2/k(T)$ \footnote{In particular, the extension $E_2/k(T)$ is not parametric over $k$.}. 

\vspace{1mm}

\noindent
Furthermore these Galois extensions of $k$ may be obtained by specializing $E_1/k(T)$.
\end{theorem}

Recall that a set $S$ of conjugacy classes of $H$ is called {\it{$g$-complete}} (a terminology due to Fried \cite{Fri95}) if no proper subgroup of $H$ intersects each conjugacy class in $S$. For instance, the set of all conjugacy classes of $H$ is g-complete \cite{Jor72}.

\vspace{3mm}

\noindent
{\it{Addendum}} 4.2. 
Under either one of the following two extra conditions:

\vspace{0.5mm}

\noindent
{\rm{(1)}} $k$ is hilbertian,

\vspace{0.5mm}

\noindent
{\rm{(2)}} there exists some subset $I \subset \{1,\dots,r_1\}$ satisfying the following two conditions:

\vspace{0.5mm}

{\rm{(a)}} $m_{i,1}(T) \cdot m^*_{i,1}(T)$ has infinitely many distinct prime divisors for 

each index $i \in I$,

\vspace{0.5mm}

{\rm{(b)}} the set $\{C_{i,1} \, / \, i \in I \}$ is g-complete,

\vspace{0.5mm}

\noindent
the following more precise non $H$-parametricity condition holds:

\vspace{1.5mm}

\noindent
({\rm{non $H$-parametricity}}) {\it{there exist infinitely many distinct Galois extensions of $k$ of group $H$ which each is not a specialization of $E_2/k(T)$.}}

\vspace{1.5mm}

\noindent
Moreover these Galois extensions of $k$ of group $H$ may be obtained by specializing $E_1/k(T)$ and, in the case the base field $k$ is assumed to be hilbertian, they may be further required to be linearly disjoint.

\vspace{2mm}

Theorem \ref{methode} is proved in $\S$4.3.

\subsection{Practical forms of theorem \ref{methode}} Continue with the notation from $\S$4.1.1. We now give four more practical forms of theorem \ref{methode}. The first one rests on a sharp variant of the Branch Point Hypothesis and the other ones each uses the Inertia Hypothesis.

\subsubsection{Branch Point Criterion} If $E_1/k(T)$ has at least one $k$-rational branch point $t_{i,1}$, then all but finitely many primes of $A$ obviously are prime divisors of $m_{i,1}(T) \cdot m_{i,1}^*(T)$, and so are of $m_{E_1}(T) \cdot m_{E_1}^*(T)$ too. Hence one obtains the following statement:

\vspace{3mm}

\noindent
{\bf{Branch Point Criterion.}} {\it{The {\rm{(non $H$-parametricity)}} condition\footnote{Here and in the next criteria, one can add as in theorem \ref{methode} that the Galois extensions of group $H$ whose existence is claimed may be obtained by specialization.} holds if the following three conditions are satisfied:

\vspace{0.5mm}

\noindent
{\rm{(BPC-1)}} $k$ is a number field,

\vspace{0.5mm}

\noindent
{\rm{(BPC-2)}} the extension $E_1/k(T)$ has at least one $k$-rational branch point,

\vspace{0.5mm}

\noindent
{\rm{(BPC-3)}} there exist infinitely many distinct primes of $A$ which each is not a prime divisor of $m_{E_2}(T) \cdot m_{E_2}^*(T)$.}}

\vspace{3mm}

An obvious necessary condition for condition (BPC-3) to hold is that $E_2/k(T)$ has no $k$-rational branch point. Moreover condition (BPC-1) may be replaced by either one of the two conditions of addendum 4.2.

\subsubsection{Inertia Criteria} Since part (b) of the Inertia Hypothesis does not depend on the base field $k$, one obtains the following three criteria in which the {\rm{(non $H$-parametricity)}} condition remains true after any finite scalar extension, {\it{i.e.}} in which the following holds:

\vspace{1.5mm}

\noindent
(geometric non $H$-parametricity) {\it{for any finite extension $k'/k$, there exist infinitely many distinct Galois extensions of $k'$ of group $H$ which each is not a specialization of $E_2k'/k'(T)$.}}

\vspace{1.5mm}

\noindent
Moreover, given a finite extension $k'/k$, these Galois extensions of $k'$ of group $H$ may be obtained by specializing $E_1k'/k'(T)$ and, in the case the base field $k$ is assumed to be hilbertian, they may be further required to be linearly disjoint.

\vspace{2.5mm}

\noindent
{\bf{Inertia Criterion 1.}} {\it{The {\rm{(geometric non $H$-parametricity)}} condition holds if the following three conditions are satisfied:

\vspace{0.5mm}

\noindent
{\rm{(IC1-1)}} each branch point of $E_1/k(T)$ is $k$-rational,

\vspace{0.5mm}

\noindent
{\rm{(IC1-2)}} there exists some index $i \in \{1,\dots,r_1\}$ such that $C_{i,1}^G$ is not contained in the set $\{C_{1,2}^{a}, \dots, C_{r_2,2}^{a} \,  /  \, a \in \mathbb{N}\}$,

\vspace{0.5mm}

\noindent
{\rm{(IC1-3)}} the set $\{C_{1,1},\dots,C_{r_1,1}\}$ is $g$-complete.}}

\vspace{2.5mm}

Indeed, given a finite extension $k'/k$, apply theorem \ref{methode} to the extensions $E_1k'/k'(T)$ and $E_2k'/k'(T)$. Fix an index $i \in \{1,\dots,r_H\}$ such that the set $\{C_{1,2}^{a}, \dots, C_{r_2,2}^{a} \,  /  \, a \in \mathbb{N}\}$ does not contain $C_{i,1}^G$ (condition (IC1-2)). Then part (b) of the Inertia Hypothesis holds for this index $i$. From condition (IC1-1), $t_{i,1}$ is $k'$-rational and then part (a) of the Inertia Hypothesis also holds for this $i$ (as noted at the beginning of $\S$4.2.1). As condition (2) of addendum 4.2 holds (with $I=\{1,\dots,r_1\}$) from conditions (IC1-1) and (IC1-3), the conclusion follows.

\vspace{2.5mm}

\noindent
{\bf{Inertia Criterion 2.}} {\it{The {\rm{(geometric non $H$-parametricity)}} condition holds if the following two conditions are satisfied:

\vspace{0.5mm}

\noindent
{\rm{(IC2-1)}} there is a $k$-rational branch point $t_{i,1}$ of $E_1/k(T)$ such that the set $\{C_{1,2}^{a}, \dots, C_{r_2,2}^{a} \,  /  \, a \in \mathbb{N}\}$ does not contain $C_{i,1}^G$,

\vspace{0.5mm}

\noindent
{\rm{(IC2-2)}} $k$ is hilbertian.}}

\vspace{3mm}

Indeed, given a finite extension $k'/k$, apply theorem \ref{methode} to the extensions $E_1k'/k'(T)$ and $E_2k'/k'(T)$. From condition (IC2-1), the Inertia Hypothesis is satisfied. As $k'$ is hilbertian from condition (IC2-2), {\it{i.e.}} condition (1) of addendum 4.2 is satisfied, the conclusion follows.

\vspace{2.5mm}

\noindent
{\bf{Inertia Criterion 3.}} {\it{The {\rm{(geometric non $H$-parametricity)}} condition holds if the following two conditions are satisfied:

\vspace{0.5mm}

\noindent
{\rm{(IC3-1)}} there exists some index $i \in \{1,\dots,r_1\}$ such that $C_{i,1}^G$ is not contained in $\{C_{1,2}^{a}, \dots, C_{r_2,2}^{a} \,  /  \, a \in \mathbb{N}\}$,

\vspace{0.5mm}

\noindent
{\rm{(IC3-2)}} $k$ is either a number field or a finite extension of a rational function field $\kappa(X)$ with $\kappa$ an arbitrary algebraically closed field of characteristic zero (and $X$ an indeterminate).}}

\vspace{2.5mm}

Indeed, given a finite extension $k'/k$, apply theorem \ref{methode} to the extensions $E_1k'/k'(T)$ and $E_2k'/k'(T)$. Fix an index $i \in \{1,\dots,r_1\}$ such that the set $\{C_{1,2}^{a}, \dots, C_{r_2,2}^{a} \,  /  \, a \in \mathbb{N}\}$ does not contain $C_{i,1}^G$ (condition (IC3-1)). Then part (b) of the Inertia Hypothesis holds for this index $i$. We show below that part (a) of the Inertia Hypothesis also holds for this $i$. As condition (1) of addendum 4.2 is satisfied from condition (IC3-2), the conclusion follows.

Indeed it follows from condition (IC3-2) that any non constant polynomial $P(T) \in k'[T]$ has infinitely many distinct prime divisors: this classically follows from the Tchebotarev density theorem ({\it{e.g.}} \cite[chapter 7, theorem (13.4)]{Neu99}) in the case $k$ is a number field (and so $k'$ is too) and this is left to the reader as an easy exercise in the function field case. 

\begin{remark} \label{forme 3}
Part (b) of the Inertia Hypothesis (and similar other of our conditions) has a stronger but more practical variant in terms of ramification indices instead of inertia canonical conjugacy classes.

Indeed, given an index $i \in \{1,\dots,r_1\}$, if the ramification index of $t_{j,2}$ in $E_2\overline{k}/\overline{k}(T)$ is a multiple of that of $t_{i,1}$ in $E_1\overline{k}/\overline{k}(T)$ for no index $j \in \{1,\dots,r_2\}$, then $\{C_{1,2}^{a}, \dots, C_{r_2,2}^{a} \,  /  \, a \in \mathbb{N}\}$ does not contain $C_{i,1}^G$.
\end{remark}

\subsection{Proof of theorem \ref{methode}} The proof rests on the main result of \cite{Leg13a} (theorem 3.1 there) and a classical result on the ramification in the specializations of a $k$-regular finite Galois extension of $k(T)$ (recalled as the ``Specialization Inertia Theorem" in \cite[$\S$2.2.3]{Leg13a}) and is similar to that of theorem 4.2 of \cite{Leg13a} (which is theorem \ref{methode} here in the special case $H=G$ and $k$ is a number field). We reproduce this proof below with the necessary adjustments for the bigger generality.

First assume that the Branch Point Hypothesis holds. Then there exists some index $i \in \{1,\dots,r_1\}$ such that the polynomial $m_{i,1}(T) \cdot m_{i,1}^*(T)$ has infinitely many distinct prime divisors $\mathcal{P}$ which each is not a prime divisor of $m_{E_2}(T)\cdot m_{E_2}^*(T)$. Furthermore, up to excluding finitely many of these primes, one may also assume that such a prime $\mathcal{P}$ satisfies the following two conditions:

\vspace{0.5mm}

\noindent
(i) $\mathcal{P}$ is a {\it{good prime for $E_1/k(T)$}} in the sense of \cite[definition 2.6]{Leg13a} and $m_{i,1}(T)$, $m_{i,1}^*(T)$ each has its coefficients in $A_\mathcal{P}$,

\vspace{0.5mm}

\noindent
(ii) $\mathcal{P}$ is a good prime for $E_2/k(T)$ and the polynomials $m_{j,2}(T)$ and $m_{j,2}^*(T)$ have their coefficients in $A_\mathcal{P}$ for each index $j \in \{1,\dots,r_2\}$.

\vspace{0.5mm}

\noindent
For such a prime $\mathcal{P}$, apply \cite[theorem 3.1]{Leg13a} to construct a specialization $F_\mathcal{P}/k$ of $E_1/k(T)$ which ramifies at $\mathcal{P}$. From \cite[corollary 2.12]{Leg13a}, $F_\mathcal{P}/k$ does not occur as a specialization of $E_2/k(T)$ and the conclusion follows. 

Now assume that the Inertia Hypothesis holds. From its part (a), there exist infinitely many distinct prime divisors $\mathcal{P}$ of $m_{i,1}(T) \cdot m^*_{i,1}(T)$ which each may be assumed as before to further satisfy conditions (i) and (ii) above. For such a prime $\mathcal{P}$, apply \cite[theorem 3.1]{Leg13a} to construct a specialization $F_\mathcal{P}/k$ of $E_1/k(T)$ whose inertia group at $\mathcal{P}$ is generated by some element of $C_{i,1}$. If $F_\mathcal{P}/k$ is a specialization of $E_2/k(T)$, then, from the Specialization Inertia Theorem, there exist some index $j \in \{1,\dots,r_2\}$ and some positive integer $a$ such that the inertia group of $F_\mathcal{P}/k$ at $\mathcal{P}$ is generated by some element of $C_{j,2}^a$. This contradicts part (b) of the Inertia Hypothesis. Hence $F_\mathcal{P}/k$ is not a specialization of $E_2/k(T)$ and the conclusion follows.

Furthermore assume that condition (2) of addendum \ref{methode} holds. Instead of \cite[theorem 3.1]{Leg13a}, use \cite[corollary 3.4 and remark 3.5]{Leg13a} in the previous two paragraphs. In each case, the extension $F_\mathcal{P}/k$ may be required to have Galois group $H$ and the (non $H$-parametricity) condition follows. In the case condition (1) holds, \cite[corollary 3.3]{Leg13a} should be used (instead of \cite[corollary 3.4 and remark 3.5]{Leg13a}) to obtain the (non $H$-parametricity) condition and the extra linearly disjointness condition.

\section{A general consequence over number fields}

Our method to obtain examples of non $G$-parametric extensions over a given base field $k$ with prescribed Galois group $G$ starts with the knowledge of two $k$-regular Galois extensions of $k(T)$ of group $G$ with some somehow incompatible ramification data. Over number fields, the state-of-the-art in inverse Galois theory does not always provide such extensions in general. Proposition \ref{prop non conj}, our conditional result, provides an inverse Galois theory assumption which makes the method work. This statement leads in particular to corollary \ref{coro non conj} which is theorem 2 from the presentation. Corollary \ref{fri}, our conjectural result, is the corresponding result under a conjecture of Fried.

\subsection{The number field case} Let $k$ be a number field and $G$ a finite group. Denote the set of all conjugacy classes of $G$ by ${\bf{cc}}(G)$. 

\subsubsection{Conditional result}

To make the rest of this section simpler, we will use the following condition:

\vspace{2mm}

\noindent
(H1/$k$) {\it{each non trivial conjugacy class of $G$ occurs as the inertia canonical conjugacy class associated with some branch point of some $k$-regular Galois extension of $k(T)$ of group $G$.}}

\vspace{2mm}

It is unknown in general if any finite group satisfies the inverse Galois theory condition (H1/$k$) for a given number field $k$. However, as recalled below, every finite group satisfies condition (H1/$k$) for suitable number fields $k$. Indeed the Riemann existence theorem classically provides the following statement ({\it{e.g.}} \cite[\S12]{Deb01a}):

\vspace{2mm}

\noindent
($*$) {\it{Any set $\{C_1,\dots,C_r\}$ of non trivial conjugacy classes of $G$ whose all elements generate $G$ occurs as the inertia canonical conjugacy class set of some Galois extension of $\overline{\mathbb{Q}}(T)$ of group $G$.}}

\vspace{2mm}

\noindent
In particular, there exists some Galois extension $\overline{E}/\overline{\mathbb{Q}}(T)$ of group $G$ whose inertia canonical conjugacy class set is the set of all non trivial conjugacy classes of $G$. Hence condition (H1/$k$) holds over any number field $k$ that is a field of definition of $\overline{E}/\overline{\mathbb{Q}}(T)$.

\begin{proposition} \label{prop non conj}
Let $E/k(T)$ be a $k$-regular Galois extension of group $G$ and inertia canonical invariant $(C_1,\dots,C_r)$. Assume that the following condition holds:

\noindent
{\rm{(H2)}} $\{C_1^a,\dots,C_r^a \, / \,  a \in \mathbb{N} \} \not= {\bf{cc}}(G)$.

\vspace{0.5mm}

\noindent
Then, under condition {\rm{(H1/$k$)}}, the extension $E/k(T)$ satisfies the {\rm{(geometric non $G$-parametricity)}} condition.
\end{proposition}

In particular, under the sole condition (H2), there exists some number field $k'$ containing $k$ such that the extension $Ek'/k'(T)$ satisfies the {\rm{(geometric non $G$-parametricity)}} condition.

\begin{proof}
Let $C$ be a conjugacy class of $G$ which is not in $\{C_1^a,\dots,C_r^a \, / \,  a \in \mathbb{N} \}$ (condition (H2)) and $E'/k(T)$ a $k$-regular Galois extension of group $G$ such that $C$ occurs as the inertia canonical conjugacy class associated with one of its branch points (condition (H1/$k$)). Then the two extensions $E'/k(T)$ and $E/k(T)$ satisfy condition (IC3-1) of Inertia Criterion 3. As condition (IC3-2) also holds, the conclusion follows.
\end{proof}

Now assume that the following group theoretical condition holds:

\vspace{2mm}

\noindent
{\it{There exists some set $\{C_1,\dots,C_r\}$ of non trivial conjugacy classes of $G$ satisfying the following two conditions:}}

\vspace{0.5mm}

\noindent
{\rm{(1)}} {\it{the elements of $C_1,\dots,C_r$ generate $G$, }} 

\vspace{0.5mm}

\noindent
{\rm{(2)}} $\{C_1^a , \dots, C_r^a\, / \,   a \in \mathbb{N} \} \not= {\bf{cc}}(G)$. 

\vspace{2mm}

\noindent
Then such a set $\{C_1,\dots,C_r\}$ occurs as the inertia canonical conjugacy class set of some $k'$-regular Galois extension $E'/k'(T)$ of group $G$ for some number field $k'$ satisfying condition (H1/$k'$) (condition (1) and statement ($*$)). Moreover the extension $E'/k'(T)$ satisfies condition (H2) of proposition \ref{prop non conj} (condition (2)). One then obtains the following:

\begin{corollary} \label{coro non conj}
There exist some number field $k'$ and some $k'$-regular Galois extension of $k'(T)$ of group $G$ satisfying the {\rm{(geometric non $G$-parametricity)}} condition.
\end{corollary}

Many finite groups satisfy the above group theoretical condition (and then the conclusion of corollary \ref{coro non conj}). Here are some of them.

\vspace{2mm}

\noindent
(a) Given two non trivial finite groups $G_1$ and $G_2$, the product $G_1 \times G_2$ does (in particular, any abelian finite group which is not cyclic of prime power order does\footnote{Note that this does not hold if $G$ is cyclic of prime power order.}).

Indeed the elements, and {\it{a fortiori}} their conjugacy classes, $(g_1,1)$ ($g_1 \in G_1$) and $(1,g_2)$ ($g_2 \in G_2$) obviously generate the product $G_1 \times G_2$. And no couple of non trivial elements $(g_1,g_2) \in G_1 \times G_2$ is conjugate to a power of one of these couples.

\vspace{2mm}

\noindent
(b) Symmetric groups $S_n$ ($n \geq 3$), alternating groups $A_n$ ($n \geq 4$) and dihedral groups $D_n$ ($n \geq 2$) obviously do.

\vspace{2mm}

\noindent
(c) Non abelian simple groups do. Indeed, as shown in \cite{Wag78} and \cite{MSW94}, such a group may be generated by involutions. Then, for any odd prime divisor $p$ of the order of the group, no element of order $p$ is conjugate to a power of an involution and the conclusion follows.

\subsubsection{Conjectural result}
By taking $\{C_1,\dots,C_r\}$ to be the set of all non trivial conjugacy classes of $G$ in the following conjecture of Fried, the inverse Galois theory condition (H1/$\mathbb{Q}$) from \S5.1.1 holds:

\vspace{1.75mm}

\noindent
Conjecture (Fried). {\it{Let $\{C_1,\dots,C_r\}$ be a set of non trivial conjugacy classes of $G$ satisfying the following two conditions:

\vspace{0.5mm}

\noindent
{\rm{(1)}} the elements of $C_1,\dots,C_r$ generate $G$,

\vspace{0.5mm}

\noindent
{\rm{(2)}} $\{C_1,\dots,C_{r}\}$ is a rational\footnote{{\it{i.e.}} $g^m \in \cup_{i=1}^{r} C_i$ for each element $g \in \cup_{i=1}^{r} C_i$ and each positive integer $m$ relatively prime to the least common multiple of the orders of the elements of $C_1,\dots,C_{r}$.} set of conjugacy classes.

\vspace{0.5mm}

\noindent
Then $\{C_1,\dots,C_r\}$ occurs as the inertia canonical conjugacy class set of some $\mathbb{Q}$-regular Galois extension of $\mathbb{Q}(T)$ of group $G$.}}

\vspace{1.75mm}

Under Fried's conjecture, one then obtains corollary \ref{fri} below:

\begin{corollary} \label{fri}
Assume that there exists some set $\{C_1,\dots,C_r\}$ of non trivial conjugacy classes of $G$ satisfying the following three conditions:

\vspace{0.5mm}

\noindent
{\rm{(1)}} the elements of $C_1,\dots,C_r$ generate $G$,

\vspace{0.5mm}

\noindent
{\rm{(2)}} $\{C_1,\dots,C_r\}$ is a rational set of conjugacy classes,

\vspace{0.5mm}

\noindent
{\rm{(3)}} $\{C_1^a , \dots, C_r^a\, / \,   a \in \mathbb{N} \} \not= {\bf{cc}}(G)$. 

\vspace{0.5mm}

\noindent
Then there exists some $\mathbb{Q}$-regular Galois extension of $\mathbb{Q}(T)$ of group $G$ satisfying the {\rm{(geometric non $G$-parametricity)}} condition.
\end{corollary}

Indeed, under Fried's conjecture, conditions (1) and (2) provide a $\mathbb{Q}$-regular Galois extension $E/\mathbb{Q}(T)$ of group $G$ whose inertia canonical conjugacy class set is $\{C_1,\dots,C_r\}$. Moreover $E/\mathbb{Q}(T)$ satisfies condition (H2) of proposition \ref{prop non conj} (condition (3)) and, as already noted, condition (H1/$\mathbb{Q}$) also holds under Fried's conjecture.

\subsection{Other base fields} Assume in this subsection that the base field $k$ is a finite extension of a rational function field $\kappa(X)$ with $\kappa$ an arbitrary algebraically closed field of characteristic zero (and $X$ an indeterminate).

In this case, condition (H1/$k$) is satisfied (statement ($*$)). Conjoining this and the proof of proposition \ref{prop non conj} shows that the conclusion of this statement holds under the sole condition (H2). Moreover, under the group theoretical condition from \S5.1.1\footnote{introduced before corollary \ref{coro non conj}.}, corollary \ref{coro non conj} holds with the suitable number field $k'$ replaced by our given base field $k$.

\section{Applications of the Branch Point Criterion}

Given a number field $k$ and a finite group $H$, we use below the Branch Point Criterion to show that some known $k$-regular finite Galois extensions of $k(T)$ of group $G$ containing $H$ satisfy the {\rm{({\rm{non $H$-parametricity}})}} condition.

\subsection{A general result} The aim of this subsection is corollary \ref{data} below. Given a field $k$ and a finite group $H$, we will use the following condition which has already appeared in \cite{Leg13a}:

\vspace{2mm}

\noindent
(H3/$k$) {\it{the group $H$ occurs as the Galois group of a $k$-regular Galois extension of $k(T)$ with at least one $k$-rational branch point.}}

\vspace{2mm}

Not all finite groups satisfy condition (H3/$k$) for a given number field $k$. For example, if $k \subset \mathbb{R}$, such a group should be of even order \cite[corollary 1.3]{DF90}. However, as recalled in \cite[$\S$4.3.1]{Leg13a}, any finite group satisfies condition (H3/$k$) for suitable number fields $k$.

\subsubsection{Statement of the result} 
Let $k$ be a number field, $G$ a finite group and $E/k(T)$ a $k$-regular Galois extension of group $G$. Denote the orbits of its branch points under the action of ${\rm{G}}_k$ by $O_1,\dots, O_s$ and, for each $i \in \{1,\dots,s\}$, the field generated over $k$ by all points in $O_i$ by $F_i$.

\begin{corollary} \label{data}
Assume that either one of the following two conditions is satisfied:

\vspace{0.5mm}

\noindent
{\rm{(1)}} $|O_i| \geq 2$ and the fields $F_1, \dots,F_s$ are linearly disjoint over $k$ \footnote{{\it{i.e.}} $F_i$ and $F_1 \dots F_{i-1} F_{i+1} \dots  F_s$ are linearly disjoint over $k$ ($i=1,\dots,s$).},

\vspace{0.5mm}

\noindent
{\rm{(2)}} $s=2$ and $|O_1|=|O_2|=2$.

\vspace{0.5mm}

\noindent
Then the extension $E/k(T)$ satisfies the {\rm{({\rm{non $H$-parametricity}})}} condition for any subgroup $H \subset G$ satisfying condition {\rm{(H3/$k$)}}.
\end{corollary}

\begin{remark} \label{cex}
Assume that $G$ satisfies condition {\rm{(H3/$k$)}} and $E/k(T)$ has $r \leq 4$ branch points. From corollary \ref{data}, we obtain that

\vspace{0.75mm}

\noindent
\hspace{0.4cm} {\it{if {\rm{(i)}} no branch point is $k$-rational,

\vspace{0.5mm}

\noindent
then {\rm{(ii)}} $E/k(T)$ satisfies the {\rm{({\rm{non $G$-parametricity}})}} condition}}.

\vspace{1mm}

\noindent
In particular, we reobtain \cite[corollary 4.4]{Leg13a} which is implication (i) $\Rightarrow$ (ii) in the special case $r=4$.

Proposition \ref{ad r=3} shows however that the converse (ii) $\Rightarrow$ (i) does not hold in general if $r=3$: the extension $E/\mathbb{Q}(T)$ there has at least one $\mathbb{Q}$-rational branch point (as noted in the proof) but condition (ii) holds. Proposition \ref{ad r=4} provides a similar counter-example in the case $r=4$.

However, for $r=2$ and number fields $k \subset \mathbb{R}$, this converse (ii) $\Rightarrow$ (i) is true. Indeed fix such a number field $k$ and assume that $E/k(T)$ has two branch points with at least one $k$-rational. Then the other one is also $k$-rational. From \cite[theorem 1.1]{DF94}, ${\rm{Gal}}(E/k(T))$ is generated by involutions and, since it is cyclic, one has ${\rm{Gal}}(E/k(T))=\mathbb{Z}/2\mathbb{Z}$. Next repeat the same argument as in the proof of (3) $\Rightarrow$ (1) in proposition \ref{ad r=2} to conclude that $E/k(T)$ is parametric over $k$.
\end{remark}

\subsubsection{Proof of corollary \ref{data}}
We show below that, under either one of conditions (1) and (2), there are infinitely many distinct primes of the integral closure $A$ of $\mathbb{Z}$ in $k$ which each is not a prime divisor of $m_E(T) \cdot m_E^*(T)$. Given a subgroup $H \subset G$ satisfying condition (H3/$k$) and a $k$-regular Galois extension $E_H/k(T)$ of group $H$ with at least one $k$-rational branch point, the conclusion then follows from the Branch Point Criterion applied to the extensions $E_H/k(T)$ and $E/k(T)$.

For each index $i \in \{1,\dots,s\}$, pick $t_i \in O_i$ and let $m_i(T)$ be the irreducible polynomial of $t_i$ over $k$ and $d_i$ the degree of $m_i(T)$. Denote the action of ${\rm{Gal}}(F_i/k)$ on the roots of $m_i(T)$ by $\sigma_i \, : \, {\rm{Gal}}(F_i/k) \rightarrow S_{d_i}$ ($i=1,\dots,s$) and the splitting field of $\prod_{i=1}^s m_i(T)$ over $k$ by $F$.

First assume that condition (1) holds. From the second part of our hypothesis, ${\rm{Gal}}(F/k)$ is isomorphic to ${\rm{Gal}}(F_1/k) \times \dots \times {\rm{Gal}}(F_s/k)$ and $\sigma_1 \times \dots \times \sigma_s : {\rm{Gal}}(F_1/k) \times \dots \times {\rm{Gal}}(F_s/k) \rightarrow S_{d_1+\dots+d_s}$ corresponds to its action on the roots of $\prod_{i=1}^s m_i(T)$. Given an index $i \in \{1,\dots,s\}$, the assumption $|O_i| \geq 2$ and an easy group theoretical lemma\footnote{Namely, given a finite group $\Gamma$ acting transitively on a finite set $S$ with at least two elements, there exists some $\gamma \in \Gamma$ such that $\gamma \, . \,  s =s$ for no $s \in S$. This can be easily obtained from the Burnside lemma \cite{Bur55}.} show that there is some $g_i \in {\rm{Gal}}(F_i/k)$ such that $\sigma_i(g_i)$ has no fixed points.

From the Tchebotarev density theorem, there exist infinitely many distinct primes of $A$ such that the associated Frobenius is conjugate in ${\rm{Gal}}(F/k)$ to $(g_1,\dots,g_s)$. In particular, there are infinitely many distinct primes of $A$ which are not prime divisors of $\prod_{i=1}^s m_i(T)$, and so not of $m_E(T)$ either. Since $\infty$ is not a branch point of $E/k(T)$, the same conclusion holds for $m_E(T)  \cdot  m_E^*(T)$ \cite[remark 3.11]{Leg13a}.

Now assume that condition (2) holds. From the last two paragraphs, we may assume that $F_1=F_2$. Then $m_{1}(T)$ and $m_{2}(T)$ have the same prime divisors up to finitely many. As $m_{1}(T)$ is irreducible over $k$ and has degree $\geq 2$, there exist infinitely many distinct primes of $A$ which each is not a prime divisor of $m_{1}(T)$ ({\it{e.g.}} \cite[theorem 9]{Hei67}), and so not of $m_{1}(T)  \cdot  m_{2}(T)$ either, thus ending the proof.

\begin{remark}
(1) As pointed out by the referee, the proof shows that the conclusion of corollary \ref{data} still holds under the sole assumption that there exists some $g \in {\rm{Gal}}(F/k)$ fixing no element of $O_1 \cup \dots \cup O_s$.

\vspace{1mm}

\noindent
(2) If $s=2$ and $|O_1| \geq 3$ or $|O_2| \geq 3$, the proof does not work in general. Indeed $P(T)=(T^3-2)(T^2+T+1)$ has zero mod $p$ for all primes $p$ \cite[$\S$7]{Nag69}.
\end{remark}

\subsection{Examples}
As already said in the presentation, $k$-regular Galois extensions of $k(T)$ with given Galois group $G$ (and {\it{a fortiori}} satisfying the assumptions of corollary \ref{data}) are not always known yet. Of course such extensions always exist in the case $G=\mathbb{Z}/2\mathbb{Z}$ (and then $H=\mathbb{Z}/2\mathbb{Z}$ too). We focus in $\S$6.2.1 on this particular situation. We next give in $\S$6.2.2 another examples with $H=G=\mathbb{Z}/2\mathbb{Z}$ and conclude in $\S$6.2.3 by some examples with larger cyclic groups.

\subsubsection{Application of corollary \ref{data}} Let $k$ be a number field and $P(T) \in k[T]$ a separable polynomial over $k$ of even degree. 

Lemma \ref{grolem} shows that the branch points of the $k$-regular quadratic extension $k(T)(\sqrt{P(T)})/k(T)$ are the roots of $P(T)$. Hence the orbits $O_1,\dots,O_s$ of corollary \ref{data} exactly correspond to the root sets of the irreducible factors $P_1(T), \dots,P_s(T)$ over $k$ of $P(T)$. Thus part (1) of corollary \ref{data} yields corollary \ref{Z/2Z} below (using instead part (2) of corollary \ref{data} leads to a weaker form of \cite[corollary 4.5]{Leg13a}):

\begin{corollary} \label{Z/2Z}
Denote the splitting fields over $k$ of the polynomials $P_1(T), \dots, P_s(T)$ by $F_1,\dots, F_s$ respectively. Assume that ${\rm{deg}}(P_i(T)) \geq 2$ for each index $i \in \{1,\dots,s\}$ and the fields $F_1, \dots, F_s$ are linearly disjoint over $k$. Then the extension $k(T)(\sqrt{P(T)})/k(T)$ satisfies the {\rm{(non $\mathbb{Z}/2\mathbb{Z}$-parametricity)}} condition.
\end{corollary}

In particular, we reobtain implication (2) $\Rightarrow$ (3) in proposition \ref{ad r=2}.

\subsubsection{Cyclotomic polynomials} Fix a positive integer $s$ and an $s$-tuple $(n_1,\dots,n_s)$ of distinct integers $\geq 3$. For each index $i \in \{1,\dots,s\}$, denote the $n_i$-th cyclotomic polynomial by $\phi_{n_i}(T)$.

\begin{corollary} \label{Z/2Z 2}
The extension $\mathbb{Q}(T)(\sqrt{\phi_{n_1}(T) \dots \phi_{n_s}(T)})/\mathbb{Q}(T)$ satisfies the {\rm{(non $\mathbb{Z}/2\mathbb{Z}$-parametricity)}} condition.
\end{corollary}

\begin{proof}
Set $E= \mathbb{Q}(T)(\sqrt{\phi_{n_1}(T) \dots \phi_{n_s}(T)})$. We show below that there are infinitely many distinct primes which each is not a prime divisor of $m_E(T)  \cdot  m_E^*(T)$. The conclusion follows from the Branch Point Criterion applied to the extensions $\mathbb{Q}(\sqrt{T})/\mathbb{Q}(T)$ (for example) and $E/\mathbb{Q}(T)$.

Since $\phi_{n_1}(T) \dots \phi_{n_s}(T)$ has even degree, $\infty$ is not a branch point of $E/\mathbb{Q}(T)$ (lemma \ref{grolem}). Hence, from \cite[remark 3.11]{Leg13a}, the two polynomials $m_E(T)  \cdot  m_E^*(T)$ and $m_E(T)$ have the same prime divisors (up to finitely many). Moreover the branch points of $E/\mathbb{Q}(T)$ are the roots of $\phi_{n_1}(T) \dots \phi_{n_s}(T)$ and, with $\varphi$ the Euler function, one then has $m_E(T) = \phi_{n_1}(T)^{\varphi(n_1)} \dots \phi_{n_s}(T)^{\varphi(n_s)}$. Since, for each index $i \in \{1,\dots,s\}$, the prime divisors of $\phi_{n_i}(T)$ are known to be exactly all primes $p$ such that $p \equiv 1 \, \, {\rm{mod}} \, \, [n_i]$ (up to finitely many), any prime divisor $p$ of $m_E(T)  \cdot  m_E^*(T)$ satisfies $p \equiv 1 \, \, {\rm{mod}} \, \, n_{i_p}$ for some index $i_p \in \{1,\dots,s\}$ (up to finitely many). From the Dirichlet theorem, there exist infinitely many distinct primes $p$ which each satisfies $p \equiv 1 \, \, {\rm{mod}} \, \, n_{i}$ for no index $i \in \{1,\dots,s\}$, thus ending the proof.
\end{proof}

\subsubsection{Larger cyclic groups}

Let $n$ be a positive integer $\geq 3$. As shown in \cite{Des95}, there exists at least one $\mathbb{Q}$-regular Galois extension of $\mathbb{Q}(T)$ of group $\mathbb{Z}/n\mathbb{Z}$ and branch point set $\{e^{2ik\pi/n} \, / \, (k,n)=1\}$. Let $E_n/\mathbb{Q}(T)$ be such an extension.

\begin{corollary} \label{cyclic 2}
The extension $E_n/\mathbb{Q}(T)$ satisfies the {\rm{(non $\mathbb{Z}/m\mathbb{Z}$-para-metricity)}} condition for any positive divisor $m$ of $n$ satisfying either one of the following two conditions:

\vspace{0.5mm}

\noindent
{\rm{(1)}} $m$ is even,

\vspace{0.5mm}

\noindent
{\rm{(2)}} $m \not \in \{1,n\}$ and, if $n \equiv 2 \, \, {\rm{mod}} \, \, 4$, $m \not=n/2$.
\end{corollary}

\begin{proof}
First assume that condition (1) holds. Then the group $\mathbb{Z}/m\mathbb{Z}$ satisfies condition (H3/$\mathbb{Q}$) \footnote{introduced at the beginning of \S6.1.} ({\it{e.g.}} \cite[lemma 3.3.5]{Leg13c}). Applying part (1) of corollary \ref{data} provides the conclusion.

Now assume that condition (2) holds. First remark that, since $\infty$ is not a branch point of $E_n/\mathbb{Q}(T)$, the polynomials $m_{E_n}(T)  \cdot  m_{E_n}^*(T)$ and $m_{E_n}(T)$ have the same prime divisors (up to finitely many) \cite[remark 3.11]{Leg13a} and, since $m_{E_n}(T)$ is a power of the $n$-th cyclotomic polynomial, these prime divisors are exactly all primes $p$ such that $p \equiv 1 \, \, {\rm{mod}} \, \, n$ (up to finitely many).

One may assume from above that $m \geq 3$. From the Dirichlet theorem, there exist infinitely many distinct primes $p$ which each satisfies $p \equiv 1 \, \, {\rm{mod}} \, \, m$ and $p \not \equiv 1 \, \, {\rm{mod}} \, \, n$. Hence, with $E_m/\mathbb{Q}(T)$ any $\mathbb{Q}$-regular Galois extension obtained as $E_n/\mathbb{Q}(T)$ but for the integer $m$ ({\it{i.e.}} $E_m/\mathbb{Q}(T)$ has Galois group $\mathbb{Z}/m\mathbb{Z}$ and branch point set $\{e^{2ik\pi/m} \, / \, (k,m)=1\}$), the original Branch Point Hypothesis of theorem \ref{methode} applied to the extensions $E_m/\mathbb{Q}(T)$ and $E_n/\mathbb{Q}(T)$ holds. As condition (1) of addendum 4.2 holds, the conclusion follows.
\end{proof}

\section{Applications of the Inertia Criteria}

For this section, let $A$ be a Dedekind domain of characteristic zero with infinitely many distinct primes and $k$ its quotient field. 

Given a finite group $H$, we use below Inertia Criteria 1-3 to show that some known $k$-regular finite Galois extensions of $k(T)$ of group $G$ containing $H$ satisfy the {\rm{(geometric non $H$-parametricity)}} condition. We first consider the case $H=S_n$ ($\S$7.1) and then the case $H=A_n$ ($\S$7.2). $\S$7.3 is devoted to some other cases $H$ is a non abelian simple group and we conclude in $\S$7.4 with the case $H$ is a $p$-group.

\subsection{The case $H=S_n$} Let $n$ be an integer $\geq 3$. The aim of this subsection is corollary \ref{coro Sn} below which gives our main examples in the situation $H=G=S_n$. The involved regular realizations of $S_n$ are recalled in $\S$7.1.1. Corollary \ref{coro Sn} is stated in $\S$7.1.2 and proved in $\S$7.1.3.

\subsubsection{Some classical regular realizations of symmetric groups} Recall that a permutation $\sigma \in S_n$ is said to have {\it{type $1^{l_1} \dots n^{l_n}$}} if, for each index $i \in \{1,\dots, n\}$, there are $l_i$ disjoint cycles of length $i$ in the cycle decomposition of $\sigma$ (for example, an $n$-cycle is of type $n^1$). Denote the conjugacy class in $S_n$ of elements of type $1^{l_1} \dots n^{l_n}$ by $[1^{l_1} \dots n^{l_n}]$.

\vspace{3mm}

\noindent
(a) {\it{Morse polynomials}}. Recall that a degree $n$ monic polynomial $M(Y)$ $\in k[Y]$ is {\it{a Morse polynomial}} if the zeroes $\beta_1, \dots, \beta_{n-1}$ of the derivative $M'(Y)$ are simple and $M(\beta_i) \not= M(\beta_j)$ for $i \not=j$. For example, $M(Y)=Y^n \pm Y$ is a Morse polynomial.

Given a degree $n$ Morse polynomial $M(Y)$, the polynomial $P(T,Y)=M(Y)-T$ provides a $k$-regular Galois extension $E_{1}/k(T)$ of group $S_n$, branch point set $\{\infty, M(\beta_1),\dots,M(\beta_{n-1})\}$ and inertia canonical invariant $([n^1],[1^{n-2} 2^1],\dots,[1^{n-2} 2^1])$. See \cite[$\S$4.4]{Ser92}.

\vspace{3mm}

\noindent
(b) {\it{Trinomial realizations}}. Let $m$, $q$, $s$ be positive integers satisfying $1 \leq m \leq n$, $(m,n)=1$ and $s(n-m)-qn=1$. The trinomial $Y^n-T^qY^m+T^s$ provides a $k$-regular Galois extension $E_{2}/k(T)$ of group $S_n$, branch point set $\{0, \infty, m^m(n-m)^{n-m}n^{-n}\}$ and inertia canonical invariant $([m^1(n-m)^1], [n^1], [1^{n-2} 2^1])$. We note for later use that the set of these three conjugacy classes of $S_n$ is g-complete \cite[$\S$2.4]{Sch00}.

\vspace{3mm}

\noindent
(c) {\it{A realization with four branch points}}. Assume that $n$ is even and $n \geq 6$. From \cite{HRD03}, there exists some $k$-regular Galois extension $E_{3}/k(T)$ of group $S_n$ with four $\mathbb{Q}$-rational branch points and inertia canonical invariant $([1^2 (n-2)^1], [1^{n-3} 3^1], [2^{(n/2)}], [1^2 2^{(n-2)/2}])$.

\subsubsection{Examples with $G=S_n$} Assume that $n \geq 4$ \footnote{As explained in \cite[remark 2.2.4(b)]{Leg13c}, corollary \ref{coro Sn} does not hold in the case $n=3$.}.

\begin{corollary} \label{coro Sn}
{\rm{(1)}} The three extensions $E_1/k(T)$ (for arbitrary $n \geq 4$), $E_2/k(T)$ (if $n \not \in \{4, 6\}$) and $E_3/k(T)$ (if $n \geq 6$ is even) each satisfies the {\rm{(geometric non $S_n$-parametricity)}} condition.

\vspace{1mm}

\noindent
{\rm{(2)}} Assume that $n=6$ and $k$ is hilbertian. Then the extension $E_2/k(T)$ satisfies the {\rm{(geometric non $S_n$-parametricity)}} condition.
\end{corollary}

\begin{remark} \label{non gen}
Fix a PAC field $\kappa$ of characteristic zero and a $\kappa$-regular Galois extension $E/\kappa(T)$ of group $S_n$ (with $n \geq 4$) provided by some degree $n$ Morse polynomial with coefficients in $\kappa$. As noted in $\S$2.3.1, the extension $E/\kappa(T)$ is $S_n$-parametric over $\kappa$. But, with $X$ an indeterminate, the extension $E(X)/\kappa(X)(T)$ is not (corollary \ref{coro Sn}). Hence $E/\kappa(T)$ is not $S_n$-generic over $\kappa$. To our knowledge, no such example was known before.
\end{remark}

\subsubsection{Proof of corollary \ref{coro Sn}} The proof has two main parts. The first one consists in showing the following general result:

\vspace{2mm}

\noindent
{\it{Let $E/k(T)$ be a $k$-regular Galois extension of group $S_n$ and inertia canonical invariant $(C_1,\dots,C_r)$. Denote the set of all integers $m$ such that $1 \leq m \leq n$ and $(m,n)=1$ by $I_n$. Then the extension $E/k(T)$ satisfies the {\rm{(geometric non $S_n$-parametricity)}} condition provided that one of the following three conditions holds:

\vspace{0.5mm}

\noindent
{\rm{(1)}} $[n^1]$ is not in $\{C_1,\dots, C_r\}$,

\vspace{0.5mm}

\noindent
{\rm{(2)}} $[m^1(n-m)^1]$ is not in $\{C_1,\dots,C_r\}$ for some $m \in I_n$,

\vspace{0.5mm}

\noindent
{\rm{(3)}} $k$ is hilbertian, $n \geq 6$ is even and $[1^2(n-2)^1]$ is not in $\{C_1,\dots, C_r\}$.

\vspace{0.5mm}

\noindent
In particular, the extension $E/k(T)$ satisfies the {\rm{(geometric non $S_n$-parametricity)}} condition if $r \leq \varphi(n)/2$ \footnote{Here and in \S7.2.3, $\varphi$ denotes the Euler function.}.}}

\vspace{2mm}

This statement is a generalization of \cite[corollary 4.7]{Leg13a} and provides theorem 3 from the presentation in the case $G=S_n$ (as $\varphi$ satisfies the classical inequality $\varphi(n) \geq \sqrt{n/2}$ ($n \geq 2$)). 

The second part of the proof consists in checking that each extension $E_i/k(T)$ ($i=1,2,3$) satisfies one of the conditions above.

\vspace{2mm}

\noindent
{\it{Part 1}}. The proof consists in each case in applying Inertia Criterion 1 (if there are no assumption on the base field $k$) or Inertia Criterion 2 (if $k$ is assumed to be hilbertian) to some suitable extension $E_j/k(T)$ ($\S$7.1.1) and the given one $E/k(T)$.

First assume that $[n^1]$ is not in $\{C_1,\dots,C_r\}$. Then $[n^1]$ is not in $\{C_1^a,\dots,C_r^a \, / \, a \in \mathbb{N}\}$ either, {\it{i.e.}} condition (IC1-2) of Inertia Criterion 1 applied to the extensions $E_{2}/k(T)$ and $E/k(T)$ holds. As conditions (IC1-1) and (IC1-3) hold (as noted in $\S$7.1.1(b)), the conclusion follows.

If $[m^1(n-m)^1]$ is not in $\{C_1,\dots,C_r\}$ for some integer $m \in I_n$, repeat the same argument with $[n^1]$ replaced by $[m^1(n-m)^1]$.

Now assume that $k$ is hilbertian, $n \geq 6$ is even and $[1^2(n-2)^1]$ is not in $\{C_1,\dots,C_r\}$. Then $[1^2(n-2)^1]$ is not in $\{C_1^a,\dots,C_r^a \, / \, a \in \mathbb{N}\}$ either, {\it{i.e.}} condition (IC2-1) of Inertia Criterion 2 applied to $E_{3}/k(T)$ and $E/k(T)$ holds. As condition (IC2-2) holds, the conclusion follows.

\vspace{2mm}

\noindent
{\it{Part 2}}. Let $i \in \{1,2,3\}$.

\vspace{0.5mm}

\noindent
(a) If $i=1$, condition (2) holds (with $m=1$). 

\vspace{0.5mm}

\noindent
(b) Assume that $i=2$. If $n \not \in \{4,6\}$, one has $\varphi(n) \geq 4$ and condition (2) holds. In the case $n=6$, condition (3) holds.

\vspace{0.5mm}

\noindent
(c) If $i=3$, condition (1) holds.

\subsection{The case $H=A_n$} Let $n$ be an integer $\geq 4$. The aim of this subsection is corollary \ref{coro An} below which gives our main examples in the situation $H=G=A_n$. The involved regular realizations of $A_n$ are recalled in $\S$7.2.1. Corollary \ref{coro An} is stated in $\S$7.2.2 and proved in $\S$7.2.3.

\subsubsection{Some classical regular realizations of alternating groups} Recall that, if the conjugacy class $[1^{l_1} \dots n^{l_n}]$ in $S_n$ of permutations of type $1^{l_1} \dots n^{l_n}$ is contained in $A_n$, then $[1^{l_1} \dots n^{l_n}]$ is a conjugacy class of $A_n$ if and only if there exists some index $p \in \{1,\dots,n\}$ such that $l_p \geq 2$ or $l_{2p} \geq 1$. Otherwise $[1^{l_1} \dots n^{l_n}]$ splits into two distinct conjugacy classes of $A_n$ which we denote below by $[1^{l_1} \dots n^{l_n}]_1$ and $[1^{l_1} \dots n^{l_n}]_2$.

\vspace{2mm}

\noindent
(a) {\it{Mestre's realizations}}. Assume that $n$ is odd. In \cite{Mes90}, Mestre produces some $k$-regular Galois extensions $E'_{1}/k(T)$ of group $A_n$ with $n-1$ branch points and inertia canonical invariant $([1^{n-3} 3^1],\dots, [1^{n-3} 3^1])$.

\vspace{2mm}

\noindent
(b) {\it{From the trinomial realizations}}. Fix a positive integer $m \leq n$ such that $(m,n)=1$. Apply the ``double group trick" \cite[lemma 4.5.1]{Ser92} to the extension $E_2/k(T)$ from $\S$7.1.1(b) to obtain a three branch point $k$-regular Galois extension $E'_2/k(T)$ of group $A_n$ and, from the {\it{branch cycle lemma}} \cite{Fri77} \cite[lemma 2.8]{Vol96}, with inertia canonical invariant

\noindent
- $([m^1(n-m)^1]_1, [m^1(n-m)^1]_2, [(n/2)^2])$ if $n$ is even,

\noindent
- $([n^1]_1, [n^1]_2,[m^1((n-m)/2)^2])$ if $n$ and $m$ are odd,

\noindent
- $([n^1]_1, [n^1]_2, [(m/2)^2(n-m)^1])$ if $n$ is odd and $m$ is even.

\noindent
Note that the branch cycle lemma shows that the branch point corresponding to the following conjugacy class (in each case) is $\mathbb{Q}$-rational:

\noindent 
- $[(n/2)^2]$ if $n$ is even,

\noindent
- $[m^1((n-m)/2)^2]$ if $n$ is odd and $m$ is odd,

\noindent
- $[(m/2)^2(n-m)^1]$ if $n$ is odd and $m$ is even.

\vspace{2mm}

\noindent
(c) {\it{From the realization with four branch points}}. Assume that $n$ is even and $n \geq 6$. As explained in \cite[$\S$3.3]{HRD03}, the extension $E_{3}/k(T)$ from $\S$7.1.1(c) induces a $k$-regular Galois extension $E'_{3}/k(T)$ of group $A_n$ with five branch points and inertia canonical invariant 

\noindent
- $([1^2((n-2)/2)^2], [1^{n-3}3^1], [1^{n-3}3^1], [1^22^{n/2}], [1^22^{n/2}])$ if $n/2$ is even,

\noindent
- $([1^2((n-2)/2)^2], [1^{n-3}3^1], [1^{n-3}3^1], [1^22^{(n-2)/2}], [1^22^{(n-2)/2}])$ otherwise.

\noindent
Note that the branch point of $E'_{3}/k(T)$ associated with $[1^2((n-2)/2)^2]$ (in each case) is $\mathbb{Q}$-rational from the branch cycle lemma (if $n \geq 8$)\footnote{In the case $n=6$, the involved branch point might not be $\mathbb{Q}$-rational as the two conjugacy classes $[1^2((n-2)/2)^2]$ and $[1^22^{(n-2)/2}]$ coincide.}.

\subsubsection{Examples with $G=A_n$} Assume that $n \geq 5$.

\begin{corollary} \label{coro An}
{\rm{(1)}} Assume that $k$ is hilbertian. Then the three extensions $E'_{1}/k(T)$ (if $n$ is odd), $E'_{2}/k(T)$ (for any $n \not= 6$) and $E'_{3}/k(T)$ (if $n$ is even) each satisfies the {\rm{(geometric non $A_n$-parametricity)}} condition.

\vspace{0.7mm}

\noindent
{\rm{(2)}} Assume that $n=6$ and $k$ is either a number field or a finite extension of a rational function field $\kappa(X)$ with $\kappa$ an arbitrary algebraically closed field of characteristic zero (and $X$ an indeterminate). Then $E'_2/k(T)$ satisfies the {\rm{(geometric non $A_n$-parametricity)}} condition.
\end{corollary}

\subsubsection{Proof of corollary \ref{coro An}} As in the case $H=S_n$, the proof has two parts. The first one consists in showing the following general result:

\vspace{2mm}

\noindent
{\it{Let $E'/k(T)$ be a $k$-regular Galois extension of group $A_n$ and inertia canonical invariant $(C_1,\dots,C_r)$. Denote the set of all integers $m$ such that $1 \leq m \leq n$ and $(m,n)=1$ by $I_n$. Then the extension $E'/k(T)$ satisfies the {\rm{(geometric non $A_n$-parametricity)}} condition provided that either one of the following two conditions holds:

\vspace{1mm}

\noindent
{\rm{(1)}} $k$ is hilbertian and one of the following four conditions holds:

\vspace{0.5mm}

{\rm{(a)}} $n$ is odd and $[m^1((n-m)/2)^2]$ is not in $\{C_1,\dots,C_r\}$ for some

odd $m \in I_n$,

\vspace{0.5mm}

{\rm{(b)}} $n$ is odd and $[(m/2)^2 (n-m)^1]$ is not in $\{C_1,\dots,C_r\}$ for some 

even $m \in I_n$,

\vspace{0.5mm}

{\rm{(c)}} $n$ is even and $[(n/2)^2]$ is not in $\{C_1,\dots,C_r\}$,

\vspace{0.5mm}

{\rm{(d)}} $n \geq 8$ is even and neither $[2^1(n-2)^1]$ nor $[1^2((n-2)/2)^2]$ is in

the set $\{C_1,\dots,C_r\}$.

\vspace{1mm}

\noindent
{\rm{(2)}} $k$ is either a number field or a finite extension of a rational function field $\kappa(X)$ with $\kappa$ an arbitrary algebraically closed field of characteristic zero (and $X$ an indeterminate) and either one of the following two conditions holds:

\vspace{0.5mm}

{\rm{(a)}} $n$ is odd and neither $[n^1]_1$ nor $[n^1]_2$ is in $\{C_1,\dots,C_r\}$,

\vspace{0.5mm}

{\rm{(b)}} $n$ is even and neither $[m^1(n-m)^1]_1$ nor $[m^1(n-m)^1]_2$ is in the

set $\{C_1,\dots, C_r\}$ for some $m \in I_n$,

\vspace{0.5mm}

{\rm{(c)}} $n=6$ and neither $[2^14^1]$ nor $[1^22^2]$ is in $\{C_1,\dots,C_r\}$.

\vspace{1mm}

\noindent
In particular, the extension $E'/k(T)$ satisfies the {\rm{(geometric non $A_n$-parametricity)}} condition if $r \leq \varphi(n)/2$ and $k$ is either a number field or a finite extension of a rational function field $\kappa(X)$ with $\kappa$ an arbitrary algebraically closed field of characteristic zero.}}

\vspace{2mm}

The second part of the proof consists in checking that each extension $E'_i/k(T)$ ($i=1,2,3$) satisfies one of the conditions above.

\vspace{2mm}

\noindent
{\it{Part 1.}} The proof is very similar to that in the case $H=G=S_n$. The only differences are that

\noindent
- the extensions of $\S$7.2.1 should be used (instead of those of $\S$7.1.1),

\noindent
- Inertia Criterion 3 should be used if $k$ is assumed to be either a number field or a finite extension of a rational function field $\kappa(X)$ with $\kappa$ an arbitrary algebraically closed field of characteristic zero. 

\noindent
We only indicate in each case which conjugacy class of $A_n$ is not in $\{C_1^a,\dots,C_r^a \, \, / \, \, a \in \mathbb{N}\}$ and which extension $E'_j/k(T)$ from $\S$7.2.1 should be used to conclude. A unified proof is given in \cite[\S3.4.2.2]{Leg13c}.

Assume that $k$ is hilbertian. In case (1)-(a), the conjugacy class $[m^1((n-m)/2)^2]$ is not in $\{C_1^a,\dots,C_r^a \, \, / \, \, a \in \mathbb{N}\}$. Then use the extension $E'_{2}/k(T)$. Replace $[m^1((n-m)/2)^2]$ by $[(m/2)^2 (n-m)^1]$ in case (1)-(b) and by $[(n/2)^2]$ in case (1)-(c). In case (1)-(d), $[1^2((n-2)/2)^2]$ is not in $\{C_1^a,\dots,C_r^a \, \, / \, \, a \in \mathbb{N}\}$. Then use the extension $E'_3/k(T)$.

Now assume that $k$ is either a number field or a finite extension of $\kappa(X)$. In case (2)-(a), either one of the classes $[n^1]_1$ and $[n^1]_2$ is not in $\{C_1^a,\dots,C_r^a \, \, / \, \, a \in \mathbb{N}\}$. Then use the extension $E'_{2}/k(T)$. Replace $[n^1]_1$ and $[n^1]_2$ by $[m^1(n-m)^1]_1$ and $[m^1(n-m)^1]_2$ in case (2)-(b). In case (2)-(c), $[1^22^2]$ is not in $\{C_1^a,\dots,C_r^a \, \, / \, \, a \in \mathbb{N}\}$. Then use $E'_3/k(T)$.

\vspace{2mm}

\noindent
{\it{Part 2}}. Let $i \in \{1,2,3\}$.

\vspace{0.5mm}

\noindent
(a) If $i=1$, condition (1)-(a) holds (with $m=1$). 

\vspace{0.5mm}

\noindent
(b) Assume that $i=2$. If $n\geq 8$ is even (resp. $n=6$), condition (1)-(d) (resp. (2)-(c)) holds. If $n$ is odd and $m \in \{1,n-1\}$ (resp. $m \not \in \{1,n-1\}$), condition (1)-(b) (resp. (1)-(a)) holds. 

\vspace{0.5mm}

\noindent
(c) If $i=3$, condition (1)-(c) holds.

\vspace{2mm}

Proceeding similarly, one may also show that the three $k$-regular Galois extensions $E_1/k(T)$, $E_2/k(T)$ and $E_3/k(T)$ of group $S_n$ from $\S$7.1.1 each satisfies the {\rm{(geometric non $A_n$-parametricity)}} condition if $k$ is hilbertian (for suitable integers $n$). This is done in \cite[\S3.4.3]{Leg13c}.

\subsection{Some other cases $H$ is a non abelian simple group} We now give some examples involving some $k$-regular Galois extensions of $k(T)$ provided by the rigidity method. We use below standard Atlas \cite{Atl} notation for conjugacy classes of finite groups.

\subsubsection{Examples with ${\rm{PSL}}_2(\mathbb{F}_p)$} Let $p \geq 5$ be a prime such that $(\frac{2}{p})=-1$ \hbox{(resp. $(\frac{3}{p})=-1$) and $E_1/k(T)$ (resp. $E_2/k(T)$) a $k$-regular Galois} extension of group ${\rm{PSL}}_2(\mathbb{F}_p)$ and inertia canonical invariant $(2A,pA,pB)$

\vspace{0.4mm}

\noindent
(resp. $(3A,pA,pB)$) \cite[propositions 7.4.3-4 and theorem 8.2.2]{Ser92}.

\vspace{0.4mm}

\begin{corollary} \label{PSL}
Assume that $k$ is hilbertian and $(-1)^{(p-1)/2}p$ is a square in $k$. \hbox{Then the extensions $E_1/k(T)$ (if $(\frac{2}{p})=-1$) and $E_2/k(T)$ (if $(\frac{3}{p})$} \hbox{$=-1$) satisfy the {\rm{(geometric non ${\rm{PSL}}_2(\mathbb{F}_p)$-parametricity)}} condition.}
\end{corollary}

\begin{proof}
Let $E/k(T)$ be a $k$-regular Galois extension of group ${\rm{PSL}}_2(\mathbb{F}_p)$ with three $k$-rational branch points and inertia canonical invariant $(2A,3A,pA)$ \cite[proposition 7.4.2 and theorem 8.2.1]{Ser92}. As $3$ does not divide $2p$ (resp. $2$ does not divide $3p$), condition (IC2-1) of Inertia Criterion 2 applied to $E/k(T)$ and $E_1/k(T)$ (resp. and $E_2/k(T)$) holds (remark \ref{forme 3}). As condition (IC2-2) holds, the conclusion follows.
\end{proof}

\subsubsection{Examples with the Monster group} Let $E_1/k(T)$ be a $k$-regular Galois extension of group the Monster group ${\rm{M}}$ with three $k$-rational branch points and inertia canonical invariant $(2A,3B,29A)$ \cite[proposition 7.4.8 and theorem 8.2.1]{Ser92}. Let $E_2/k(T)$ be a $k$-regular Galois extension of group M with three $k$-rational branch points and corresponding ramification indices 2, 3, 71 \cite{Tho84} (if $-71$ is a square in $k$). Applying twice Inertia Criterion 2 (and remark \ref{forme 3}) to these extensions leads to corollary \ref{Monstre} below:

\begin{corollary} \label{Monstre} 
Assume that $k$ is hilbertian and $-71$ is a square in $k$. Then the two extensions $E_1/k(T)$ and $E_2/k(T)$ each satisfies the {\rm{(geometric non ${\rm{M}}$-parametricity)}} condition. 
\end{corollary}

\subsubsection{Examples with $H \not=G$} Let $E/k(T)$ be a $k$-regular Galois extension of group the Baby-Monster B and inertia canonical invariant $(2C,3A,55A)$ \cite[chapter II, proposition 9.6 and chapter I, theorem 4.8]{MM99}.

\begin{corollary} \label{baby}
Assume that $k$ is hilbertian. Then, with ${\rm{Th}}$ the Thompson group, the extension $E/k(T)$ satisfies the {\rm{(geometric non ${\rm{Th}}$-para-metricity)}} condition.
\end{corollary}

\begin{proof}
It suffices to apply Inertia Criterion 2 (and remark \ref{forme 3}) to the extensions $E'/k(T)$ and $E/k(T)$, where $E'/k(T)$ denotes any $k$-regular Galois extension of group ${\rm{Th}}$ with three $k$-rational branch points and inertia canonical invariant $(2A,3A,19A)$ \cite[chapter II, proposition 9.5 and chapter I, theorem 4.8]{MM99}.
\end{proof}

Any finite group $H$ is a subgroup of $G=S_n$ provided that $n \geq |H|$. This allows us to give some examples of non $H$-parametric extensions of group $S_n$ for large enough integers $n$. For instance, the extension $E_{1}/k(T)$ from $\S$7.1.1(a) satisfies the following:

\begin{corollary} \label{coro J2}
Let $n$ be a positive integer $\geq 604800$. Assume that either one of the following two conditions holds:

\vspace{0.5mm}

\noindent
{\rm{(1)}} $7 {\not \vert} n$ and $k$ is hilbertian,

\vspace{0.5mm}

\noindent
{\rm{(2)}} $5 {\not \vert} n$ and $k$ is either a number field or a finite extension of a rational function field $\kappa(X)$ with $\kappa$ an arbitrary algebraically closed field of characteristic zero (and $X$ an indeterminate).

\vspace{0.5mm}

\noindent
Then, with ${\rm{J}}_2$ the Hall-Janko group, the extension $E_{1}/k(T)$ satisfies the {\rm{(geometric non ${\rm{J_2}}$-parametricity)}} condition.
\end{corollary} 

\begin{proof}
It suffices to apply Inertia Criterion 2 if condition (1) holds or Inertia Criterion 3 if condition (2) holds (and remark \ref{forme 3} in both situations) to the extensions $E/k(T)$ and $E_{1}/k(T)$, where $E/k(T)$ denotes any $k$-regular Galois extension of group ${\rm{J}}_2$, inertia canonical invariant $(5A,5B,7A)$ and such that the branch point corresponding to $7A$ is $k$-rational \cite[proposition 7.4.7 and theorem 8.2.2]{Ser92}.
\end{proof}

\subsection{The case $H$ is a $p$-group}
Let $G$ be a finite group, $p$ a prime divisor of $|G|$ and $E/k(T)$ a $k$-regular Galois extension of group $G$.

\begin{corollary} \label{cyclic} 
Assume that the following two conditions hold:

\vspace{0.5mm}

\noindent
{\rm{(1)}} $p$ divides none of the ramification indices of the branch points,

\vspace{0.75mm}

\noindent
{\rm{(2)}} $k$ is a number field or a finite extension of a rational function field $\kappa(X)$ with $\kappa$ an arbitrary algebraically closed field of characteristic zero.

\vspace{1mm}

\noindent
Then the extension $E/k(T)$ satisfies the {\rm{(geometric non $H$-parametrici-ty)}} condition for any $p$-subgroup $H \subset G$ which occurs as the Galois group of a $k$-regular Galois extension of $k(T)$. Furthermore condition (2) can be removed in the case $p=2$ and $H=\mathbb{Z}/2\mathbb{Z}$.
\end{corollary}

\begin{remark}
Assume that $k$ is a number field. Under condition (1), one has the following two conclusions.

\vspace{1mm}

\noindent
(a) {\it{The extension $E/k(T)$ satisfies the {\rm{(geometric non $\mathbb{Z}/p\mathbb{Z}$-parametri-city)}} condition}}.

\vspace{1mm}

\noindent
(b) {\it{There is a finite extension $k'/k$ such that $Ek'/k'(T)$ satisfies the {\rm{(geometric non $H$-parametricity)}} condition for any $p$-subgroup $H \subset G$.}}

\vspace{1mm}

\noindent
In the case $k$ is a finite extension of a rational function field $\kappa(X)$ with $\kappa$ an arbitrary algebraically closed field of characteristic zero, then, under condition (1), statement (b) holds with $k'=k$.
\end{remark}

Corollary \ref{cyclic} may be applied to various $k$-regular Galois extensions of $k(T)$. For example, consider those of group the Conway group ${\rm{{Co}}}_1$ and inertia canonical invariant $(3A, 5C, 13A)$ \cite[chapter II, proposition 9.3 and chapter I, theorem 4.8]{MM99}. The set of prime divisors of $|{\rm{{Co}}}_1|$ is $\{2,3,5,7,11,13,23\}$ and condition (1) holds for any prime $p$ in $\{2,7,11,23\}$. Moreover many $k$-regular Galois extensions of $k(T)$ recalled in this paper also satisfy condition (1) (for suitable primes $p$).

\begin{proof}
Given a $p$-subgroup $H \subset G$ as in corollary \ref{cyclic}, the conclusion follows from Inertia Criterion 3 (and remark \ref{forme 3}) applied to any $k$-regular Galois extension $E_H/k(T)$ of group $H$ and $E/k(T)$. 

In the special case $p=2$ and $H=\mathbb{Z}/2\mathbb{Z}$, take $E_H=k(\sqrt{T})$ and use Inertia Criterion 1 (instead of Inertia Criterion 3) to conclude.
\end{proof}

\bibliography{Biblio2}
\bibliographystyle{alpha}

\end{document}